\RequirePackage{fix-cm}
\documentclass{amsart} 
\usepackage{mathptmx}
\usepackage{latexsym}
\usepackage{amsmath}
\usepackage{amssymb}
\usepackage{amsfonts} 
\usepackage{eucal} 
\usepackage{amsbsy}
\usepackage{amsthm}
\usepackage[all]{xy}

\newtheorem{thm}{Theorem}[section]
\newtheorem*{thm*}{Theorem}
\newtheorem{lemma}[thm]{Lemma}
\newtheorem{prop}[thm]{Proposition}
\newtheorem{cor}[thm]{Corollary}
\newtheorem*{cor*}{Corollary}

\theoremstyle{definition}
\newtheorem{defn}[thm]{Definition}
\newtheorem{example}[thm]{Example}

\theoremstyle{remark}
\newtheorem{remark}[thm]{Remark}

\newcommand {\Ha}    {\ensuremath{\mbox{$\mathcal{H}$}}}

\newcommand {\Va}    {\ensuremath{\mbox{$\mathcal{V}$}}}

\newcommand {\der}   {\ensuremath{\mathbf{H}}}

\newcommand {\bG}    {\ensuremath{\mathbf{G}}}

\newcommand {\bL}    {\ensuremath{\textbf{L}}}

\newcommand {\real}  {\ensuremath{\mathbb{R}}}
\newcommand {\intg}  {\ensuremath{\mathbb{Z}}}

\newcommand {\cplx}  {\ensuremath{\mathbb{C}}}
\newcommand {\rat}   {\ensuremath{\mathbb{Q}}}

\newcommand {\im}    {\operatorname{im}}
\newcommand {\rk}    {\operatorname{rk}}

\newcommand {\ip}    {\ensuremath{I^{\bar{p}}}}

\newcommand {\pt}    {\ensuremath{\operatorname{pt}}}
\newcommand {\id}    {\ensuremath{\operatorname{id}}}

\newcommand {\Harm}  {\ensuremath{\operatorname{Harm}}}

\newcommand {\ms}    {\ensuremath{\mathcal{MS}}}

\newcommand {\oms}    {\ensuremath{\Omega_{\ms}}}

\newcommand {\ft}   {\ensuremath{\operatorname{ft}}}
\newcommand {\fU}   {\ensuremath{\mathfrak{U}}}
\newcommand {\Der}   {\ensuremath{\operatorname{Der}}}
\newcommand {\Princ}  {\ensuremath{\operatorname{Princ}}}
\newcommand {\cd}   {\ensuremath{\operatorname{cd}}}

\newcommand {\oR}   {\ensuremath{\overline{R}}}
\newcommand {\oE}   {\ensuremath{\overline{E}}}
\newcommand {\oX}   {\ensuremath{\overline{X}}}
\newcommand {\oJ}   {\ensuremath{\overline{J}}}
\newcommand {\oF}   {\ensuremath{\overline{F}}}

\begin{document}


\title{Isometric Group Actions and the Cohomology of Flat Fiber Bundles}

\author{Markus Banagl}

\address{Mathematisches Institut, Universit\"at Heidelberg,
  Im Neuenheimer Feld 288, 69120 Heidelberg, Germany}

\email{banagl@mathi.uni-heidelberg.de}

\thanks{The author was in part supported by a research grant of the
 Deutsche Forschungsgemeinschaft.}

\date{May, 2011}




\begin{abstract}
Using methods originating in the theory of intersection spaces, specifically a
de Rham type description of the real cohomology of these spaces by a complex
of global differential forms, we show that the Leray-Serre spectral sequence
with real coefficients of a flat fiber bundle of smooth manifolds collapses if
the fiber is Riemannian and the structure group acts isometrically. The proof is
largely topological and does not need a metric on the base or total space.
We use this result to show further that if the fundamental group of a smooth
aspherical manifold acts isometrically on a Riemannian manifold, then the equivariant real
cohomology of the Riemannian manifold can be computed as a direct sum over the cohomology
of the group with coefficients in the (generally twisted) cohomology modules of
the manifold. Our results have consequences for the Euler class of flat sphere bundles.
Several examples are discussed in detail, for instance an action
of a free abelian group on a flag manifold.
\end{abstract}

\maketitle


\tableofcontents


\section{Introduction}

A fiber bundle with given structure group is \emph{flat}, if the transition
functions into the structure group are locally constant.
We show that the method of intersection spaces introduced in
\cite{banagl-intersectionspaces}, specifically the de Rham description
of intersection space cohomology given in \cite{banagl-derhamintspace},
implies the following result, by a concise and topological proof:

\begin{thm*} (See Theorem \ref{thm.main1}.)
Let $B, E$ and $F$ be closed, smooth manifolds with $F$ oriented.
Let $\pi: E\to B$ be a flat, smooth fiber bundle with
structure group $H$. If \\

(1) $H$ is a Lie group acting properly (and smoothly) on $F$ (for example $H$ compact), \\

\noindent or \\

(2) $F$ is Riemannian and the images of the monodromy homomorphisms \\
\mbox{ } \hspace{.6cm} $\pi_1 (B,b)\to H$ act by isometries on $F$, where the base-point $b$ ranges \\
\mbox{ } \hspace{.6cm} over the connected components of $B$, \\

\noindent then the cohomological Leray-Serre
spectral sequence of $\pi$ for real coefficients collapses at the $E_2$-term.
In particular, the formula
\[ H^k (E;\real) \cong \bigoplus_{p+q=k} H^p (B; \der^q (F;\real)) \]
holds, where the $\der^q (F;\real)$ are local coefficient systems on $B$
induced by $\pi$, whose groups are the real cohomology groups of the fiber.
\end{thm*}
By a result of R. Palais (\cite{palaisexistslices}),
condition (1) implies that $F$ can be endowed with a Riemannian metric
such that (2) holds. The isometry group of an $m$-dimensional, compact, 
Riemannian manifold
is a compact Lie group of dimension at most $\frac{1}{2}m(m+1)$.
Thus (2) implies (1) by taking $H$ to be the isometry group of $F$,
and the two conditions are essentially equivalent.
In Section \ref{sec.nonisometric}, we provide an example of a flat, smooth
fiber bundle whose structure group does not act isometrically for any
Riemann metric on the fiber, and whose spectral sequence does not collapse
at $E_2$. Hence, conditions (1), (2) in the theorem cannot be deleted
without substitution. \\

Let us indicate two immediate consequences of the theorem.
Let $\pi: E\to B$ be an oriented, flat sphere bundle with structure group
$\operatorname{SO}(n)$ over a closed, smooth manifold $B$. 
The transgression
\[ d_n: E^{0,n-1}_n = H^0 (B;\der^{n-1}(S^{n-1};\real)) = H^{n-1}(S^{n-1};\real) \]
\[ \hspace{3cm}\longrightarrow E^{n,0}_n = H^n (B;\der^0 (S^{n-1};\real)) = H^{n} (B;\real) \]
sends a certain element $\sigma \in E^{0,n-1}_1$, which corresponds to 
local angular forms on the sphere bundle, and which survives to $E_n$,
to the Euler class of the sphere bundle. Since $\operatorname{SO}(n)$ is
compact, the spectral sequence of the bundle collapses at $E_2$, by the theorem.
Thus $d_n=0$ and we obtain the following corollary:
\begin{cor*}
The real Euler class of an oriented, flat, linear sphere bundle (structure group
$\operatorname{SO}(n)$) over a closed, smooth manifold is zero.
\end{cor*}
We thus obtain a topological proof, without using the Chern-Weil theory of curvature forms,
of a result closely related to the results of \cite[Section 4]{milnorconncurvzero},
which do rely on Chern-Weil theory.

By \cite{smillie}, there is a flat manifold $M^{2n}$ with nonzero Euler characteristic for every
$n>1$. If the associated tangent sphere bundle of $M$, with structure group 
$\operatorname{SO}(2n)$, were flat, then the Euler class of such a sphere bundle would
vanish according to the above corollary. We arrive at:
\begin{cor*}
For every $n>1$, there is a flat manifold $M^{2n}$ whose linear tangent sphere
bundle is not flat.
\end{cor*}
For the tangent bundles of these manifolds $M^{2n},$ the structure group
$\operatorname{GL}^+ (2n,\real)$ can be reduced to $\operatorname{GL}^+ (2n,\real)$
with the discrete topology and can also be reduced to $\operatorname{SO}(2n)$
(by using a metric), but there are no further reductions to $\operatorname{SO}(2n)$
with the discrete topology.

Our theorem may be applied to the equivariant cohomology $H^\bullet_G (-;\real)$
of certain discrete groups $G$:

\begin{thm*} (See Theorem \ref{thm.equivcohom}.)
Let $F$ be an oriented, closed, smooth manifold and $G$ a discrete group,
whose Eilenberg-MacLane space $K(G,1)$ may be taken to be a closed,
smooth manifold. If for a smooth action of $G$ on $F$, \\

(1) the action factors through a proper, smooth Lie group action, \\

\noindent or \\

(2) $F$ is Riemannian and $G$ acts isometrically on $F$, \\

\noindent then there is a decomposition
\[ H^k_G (F;\real) \cong \bigoplus_{p+q=k} H^p (G; \der^q (F;\real)), \]
where the $\der^q (F;\real)$ are $G$-modules determined by the action.
\end{thm*}

The assumption on $G$ can be paraphrased as requiring $G$ to be the
fundamental group of a closed, (smooth) aspherical manifold. Examples of
such groups include finitely generated free abelian groups, the fundamental
groups of closed manifolds with non-positive sectional curvature, the
fundamental groups of surfaces other than the real projective plane,
infinite fundamental groups of irreducible, closed, orientable
$3$-manifolds, torsionfree discrete subgroups of almost connected
Lie groups, and certain groups arising from Gromov's hyperbolization technique.
Only torsionfree $G$ can satisfy the hypothesis of the above theorem.
Note that in (1) we do not assume that the image of $G$ is closed in the intermediary
Lie group, nor that one can identify $G$ with a subgroup of the
intermediary group, and in (2)
we do not assume that the image of $G$ is closed in the isometry group of $F$,
nor that one can identify $G$ with a subgroup of the isometries.
For instance, the integers $G=\intg$ with $K(\intg,1)=S^1,$ the circle,
act isometrically and freely (and ergodically) on the unit circle by powers of a rotation by
an angle which is an irrational multiple of $2\pi$. Theorem \ref{thm.equivcohom} thus
emphasizes the discrete dynamical systems viewpoint. This example
also satisfies (1), since the powers of the irrational rotation are a subgroup
(which is not closed) of the compact Lie group $S^1$ which acts on itself
by (e.g. left) multiplication. The actions to which the theorem applies
need not be proper, nor need our $G$-spaces be $G$-CW complexes, but
in many geometric situations, nonproper actions factor through proper actions
in a natural way. (The above $\intg$-action on the circle is not proper, since
e.g. the orbit space is not Hausdorff and orbits are not closed in $S^1$.)
Section \ref{sec.equivcohom} contains a number of corollaries to,
applications of, and examples illustrating Theorem \ref{thm.equivcohom}.
For instance, if an integral Heisenberg group $\mathfrak{H}_n$ acts
isometrically on an oriented, closed, connected, Riemannian manifold $F$, then
\[ \rk H^k_{\mathfrak{H}_n} (F) \geq 2,~ k= 1,2, \]
and $H^3_{\mathfrak{H}_n} (F;\real)$ does not vanish (Corollary \ref{cor.5}). \\

In Section \ref{sec.flag}, we illustrate the use of our results by calculating
explicitly an equivariant cohomology group of a certain action of a free abelian
group on the flag manifold $F_8 = U(8)/T^8$ of real dimension $56$. The outcome
is verified by an alternative, logically independent, computation of the same group,
relying on a recursive scheme introduced in the appendix (Section \ref{appendix}).
This recursive scheme is only available for finitely generated free abelian groups. \\

The proof of Theorem \ref{thm.main1} relies on the complex of multiplicatively structured
forms constructed in \cite{banagl-derhamintspace}, and the fact that fiberwise 
truncation \emph{in both directions},
yielding a \emph{subcomplex} in both cases (and not a quotient
complex in one case), can be performed on the multiplicatively structured forms.
Such truncations, carried out homotopy theoretically on the space level
(generally a much harder problem), are also required to build the
intersection space $\ip X$ of a stratified pseudomanifold $X$.
The only analytic tool we need to prove the theorem is the classical
Hodge decomposition --- hence the assumption that the fiber must be
closed, oriented and Riemannian.
Otherwise, our argument is purely topological in nature, using \v{C}ech complexes.
No connections, tensor fields, etc. are used on the base or total space;
in particular we need not assume that $B$ and $E$ are Riemannian.
The combinatorial nature of our proof may lend itself to an extension of
our theorem to nonsmoothable PL manifolds $B$ and $E$ and flat 
PL bundles $\pi: E\to B$. In this situation, smooth forms have to be replaced
by Sullivan's complex $\widetilde{A}^\bullet (K)$ of piecewise
$C^\infty$-forms on a simplicial complex $K$. We may also recall at this point
that there are
closed, aspherical PL manifolds which are not homotopy equivalent to
closed, smooth manifolds, \cite{davishausmann}.\\

Our Theorem \ref{thm.main1} is closely related to results of \cite{dainonmult} and 
\cite{muellerjoern}. Dai and M\"uller work with Riemannian $E,B$
and a Riemannian submersion $\pi: E\to B$. Their metric on $F$ is allowed, 
to a certain extent, to vary with points in the base.
Using Dai's spectral sequence degeneration result from
\cite{dainonmult}, M\"uller proves that if a flat Riemannian submersion $\pi$
is locally a warped product, or has totally geodesic fibers, then the
spectral sequence of $\pi$ collapses at $E_2$. Let us put this into perspective.
A Riemannian submersion whose total space is complete is a locally trivial
fiber bundle. The geometry of a Riemannian submersion is largely governed
by two tensor fields $T$ and $A$. Let $\Va$ and $\Ha$ denote the vertical
and horizontal distributions, respectively, that is, at each point $x\in E$ there is
an orthogonal decomposition $\Va_x \oplus \Ha_x = T_x E$ of the tangent space
and $\Va_x$ is tangent to the fiber over $\pi (x)$. Let $\Va$ and $\Ha$ also
denote the projection of a vector onto $\Va_x$ and $\Ha_x,$ respectively.
With $\nabla$ the Levi-Civita connection of the metric on $E$, one sets for
vector fields $V,W$ on $E$,
\[ T_V W = \Ha \nabla_{\Va V} \Va W + \Va \nabla_{\Va V} \Ha W \]
and
\[ A_V W = \Ha \nabla_{\Ha V} \Va W + \Va \nabla_{\Ha V} \Ha W. \]
If $V,W$ are vertical, then $T_V W$ is the second fundamental form of each
fiber. The identical vanishing $T=0$ is equivalent to each fiber being totally
geodesic, that is, geodesics in the fibers are also geodesics for $E$.
This implies that all fibers are isometric and the holonomy group
(which agrees with the structure group of $\pi$, at least when $E$ is 
complete) is a subgroup of the isometry group of the fiber. The identical
vanishing $A=0$ is equivalent to the integrability of $\Ha$. If $\Ha$ is
integrable, then $E$ is locally isometric to $B\times F$ with a metric
$g_B + g_{F,b},$ where $g_B$ is a metric on $B$ and $\{ g_{F,b} \}_{b\in B}$
is a smooth family of metrics on $F$. In this situation, the horizontal
foliation yields a flat (Ehresmann) connection for $\pi$. For a flat connection,
the holonomy along a path depends only on the homotopy class of the path.
Indeed, for path-connected $B$, flat bundles with structure group $G$
acting effectively on a fiber $G$ are in one-to-one correspondence with
homomorphisms $\pi_1 (B)\to G$. In particular, one may take $G$ to be discrete.
A \emph{warped} product is a Riemannian manifold $B\times F$, whose
metric has the form $g_B + fg_F$, where $g_F$ is a fixed metric on $F$ and
$f:B\to \real$ is a positive function. If the Riemannian submersion $\pi$ is
locally a warped product, then $\Ha$ is integrable.
When $T \equiv 0$ and $A \equiv 0$, the total space $E$ is locally a product
$g_B + g_F$, where $g_F$ does not depend on points in the base. 
From \cite{dainonmult}, M\"uller isolates a technical condition, called condition ``(B)''
in \cite[Section 2.3]{muellerjoern}, which for a flat Riemannian submersion
implies collapse of the spectral sequence at $E_2$. He shows that this condition
is satisfied for locally warped products as well as for totally geodesic fibers,
while Dai shows that collapse at $E_2$ happens for
flat Riemannian submersions satisfying (B).

\section{A Complex of Multiplicatively Structured Forms on Flat Bundles}
\label{sec.fibharmflat}

This section reviews the multiplicatively structured form model introduced in
\cite{banagl-derhamintspace}. The proofs of the cited results can be found in that paper.
For a smooth manifold $M$, $\Omega^\bullet (M)$ denotes the de Rham complex
of smooth differential forms on $M$.
Let $F$ be a closed, oriented, Riemannian manifold and $\pi: E\rightarrow B$
a flat, smooth fiber bundle over the closed, smooth $n$-dimensional base manifold $B$ with
fiber $F$ and structure group the isometries of $F$. An open cover of an
$n$-manifold is called \emph{good}, if all nonempty finite intersections of
sets in the cover are diffeomorphic to $\real^n$. Every smooth manifold
has a good cover and if the manifold is compact, then the cover can be
chosen to be finite. Let $\mathfrak{U} = \{ U_\alpha \}$ be a finite good open
cover of the base $B$ such that $\pi$ trivializes with respect to $\mathfrak{U}$.
Let $\{ \phi_\alpha \}$ be a system of local trivializations, that is, the
$\phi_\alpha$ are diffeomorphisms such that
\[ \xymatrix{
\pi^{-1} (U_\alpha) \ar[rr]^{\phi_\alpha} \ar[rd]_{\pi|} 
& & U_\alpha \times F \ar[ld]^{\pi_1} \\
& U_\alpha & } \]
commutes for every $\alpha$. Flatness implies that the transition functions
\[ \rho_{\beta \alpha} = \phi_\beta| \circ \phi_\alpha|^{-1}:
 (U_\alpha \cap U_\beta)\times F \longrightarrow
  \pi^{-1} (U_\alpha \cap U_\beta) \longrightarrow 
  (U_\alpha \cap U_\beta)\times F \]
are of the form
\[ \rho_{\beta \alpha} (t,x) = (t, g_{\beta \alpha} (x)). \]
The maps $g_{\beta \alpha}: F\rightarrow F$ are isometries.

If $X$ is a topological
space, let $\pi_2: X\times F \rightarrow F$ denote the second-factor
projection. Let $V \subset B$ be a $\mathfrak{U}$-small open subset
and suppose that $V\subset U_{\alpha}$.
\begin{defn}
A differential form $\omega \in \Omega^q (\pi^{-1} (V))$ is called
\emph{$\alpha$-multiplicatively structured}, if it has the form
\[ \omega = \phi^\ast_\alpha \sum_j \pi^\ast_1 \eta_j \wedge
  \pi^\ast_2 \gamma_j,~ 
 \eta_j \in \Omega^\bullet (V),~ \gamma_j \in \Omega^\bullet (F) \]
(finite sums).
\end{defn}

Flatness is crucial for the following basic lemma.
\begin{lemma} \label{lem.abinvariant}
Suppose $V \subset U_\alpha \cap U_\beta.$ Then $\omega$ is
$\alpha$-multiplicatively structured if, and only if, $\omega$ is $\beta$-multiplicatively structured.
\end{lemma}
The lemma follows from the \emph{transformation law}
\begin{equation} \label{equ.transformlaw}
\phi^\ast_\alpha \sum_j \pi^\ast_1 \eta_j \wedge
  \pi^\ast_2 \gamma_j 
 =  \phi^\ast_\beta \sum_j \pi^\ast_1 \eta_j \wedge
  \pi^\ast_2 (g^\ast_{\alpha \beta} \gamma_j).
\end{equation}

The lemma shows that the property of being
multiplicatively structured over $V$ is invariantly defined, independent
of the choice of $\alpha$ such that $V\subset U_\alpha$. 
We will use the shorthand notation
\[ U_{\alpha_0 \ldots \alpha_p} = U_{\alpha_0} \cap \cdots \cap U_{\alpha_p} \]
for multiple intersections. (Repetitions of indices are allowed.) 
Since $\mathfrak{U}$ is a good cover, every
$U_{\alpha_0 \ldots \alpha_p}$ is diffeomorphic to $\real^n,$ $n = \dim B$.
A linear subspace, the subspace of \emph{multiplicatively structured forms}, of $\Omega^q (E)$
is obtained by setting
\[ \Omega^q_{\ms} (B) = \{ \omega \in \Omega^q (E) ~|~ \omega|_{\pi^{-1} U_\alpha} \text{ is }
   \alpha \text{-multiplicatively structured for all } \alpha \}. \]
The exterior derivative $d: \Omega^q (E) \rightarrow \Omega^{q+1}
(E)$ restricts to a differential
\[ d: \Omega^q_\ms (B) \longrightarrow \Omega^{q+1}_\ms (B). \]
Thus $\Omega^\bullet_\ms (B) \subset \Omega^\bullet (E)$ is a 
subcomplex. We shall eventually see that this inclusion is a quasi-isomorphism, that is,
induces isomorphisms on cohomology. For any $\alpha$, set
\[ \Omega^\bullet_{\ms} (U_\alpha) = \{ \omega \in \Omega^\bullet (\pi^{-1} U_\alpha) 
  ~|~ \omega \text{ is } \alpha \text{-multiplicatively structured } \}. \]
Let $r$ denote the obvious restriction map
\[ r: \Omega^\bullet_\ms (B) \longrightarrow \prod_\alpha \Omega^\bullet_\ms (U_\alpha). \]
If $p$ is positive, then we set
\[ \Omega^\bullet_{\ms} (U_{\alpha_0 \ldots \alpha_p}) = \{ \omega \in \Omega^\bullet (\pi^{-1} U_{\alpha_0 \ldots \alpha_p}) 
  ~|~ \omega \text{ is } \alpha_0 \text{-multiplicatively structured } \}. \]
Lemma \ref{lem.abinvariant} implies that for any $1\leq j \leq p,$
\[ \Omega^\bullet_{\ms} (U_{\alpha_0 \ldots \alpha_p}) = \{ \omega \in \Omega^\bullet (\pi^{-1} U_{\alpha_0 \ldots \alpha_p}) 
  ~|~ \omega \text{ is } \alpha_j \text{-multiplicatively structured } \}. \]
In particular, if $\sigma$ is any permutation of $0,1,\ldots, p$, then
\[ \Omega^\bullet_\ms (U_{\alpha_{\sigma (0)} \ldots \alpha_{\sigma (p)}}) =
  \Omega^\bullet_\ms (U_{\alpha_0 \ldots \alpha_p}). \]
The components of an element 
\[ \xi \in \prod_{\alpha_0, \ldots, \alpha_p} \Omega^\bullet_\ms (U_{\alpha_0 \ldots
\alpha_p}) \]
will be written as
\[ \xi_{\alpha_0 \ldots \alpha_p} \in \Omega^\bullet_\ms (U_{\alpha_0 \ldots
\alpha_p}). \]
We impose the antisymmetry restriction 
$\xi_{\ldots \alpha_i \ldots \alpha_j \ldots} = - \xi_{\ldots \alpha_j \ldots \alpha_i \ldots}$
upon interchange of two indices. In particular, if $\alpha_0,\ldots, \alpha_k$ contains a repetition,
then $\xi_{\alpha_0 \ldots \alpha_k}=0.$
The difference operator
\[ \delta: \prod \Omega^\bullet (\pi^{-1} U_{\alpha_0 \ldots \alpha_p}) \longrightarrow
\prod \Omega^\bullet (\pi^{-1} U_{\alpha_0 \ldots \alpha_{p+1}}), \]
defined by
\[ (\delta \xi)_{\alpha_0 \ldots \alpha_{p+1}} = \sum_{j=0}^{p+1} (-1)^j
 \xi_{\alpha_0 \ldots \hat{\alpha}_j \ldots \alpha_{p+1}}|_{\pi^{-1}
  U_{\alpha_0 \ldots \alpha_{p+1}}} \]
and satisfying $\delta^2 =0,$ restricts to a difference operator
\[ \delta: \prod \Omega^\bullet_\ms (U_{\alpha_0 \ldots \alpha_p}) \longrightarrow
\prod \Omega^\bullet_\ms (U_{\alpha_0 \ldots \alpha_{p+1}}). \]
Since the de Rham differential $d$ commutes with restriction to open subsets, we have
$d\delta = \delta d$. Thus
\[ C^p (\mathfrak{U}; \Omega^q_\ms) = \prod \Omega^q_\ms (U_{\alpha_0 \ldots
 \alpha_p}) \]
is a double complex with horizontal differential $\delta$ and vertical differential $d$.
The associated simple complex $C^\bullet_\ms (\mathfrak{U})$ has groups
\[ C^j_\ms (\mathfrak{U}) = \bigoplus_{p+q=j} C^p (\mathfrak{U}; \Omega^q_\ms) \]
in degree $j$ and differential $D = \delta + (-1)^p d$ on $C^p (\mathfrak{U}; \Omega^q_\ms)$.
We shall refer to the double
complex $(C^\bullet (\mathfrak{U}; \Omega^\bullet_\ms), \delta, d)$ as the
\emph{multiplicatively structured \v{C}ech-de Rham complex}.

\begin{lemma} \label{lem.gmv} (Generalized Mayer-Vietoris sequence.)
The sequence
\[ 0 \longrightarrow \Omega^\bullet_\ms (B)
 \stackrel{r}{\longrightarrow} C^0 (\mathfrak{U}; \Omega^\bullet_\ms)
 \stackrel{\delta}{\longrightarrow} C^1 (\mathfrak{U}; \Omega^\bullet_\ms)
 \stackrel{\delta}{\longrightarrow} C^2 (\mathfrak{U}; \Omega^\bullet_\ms)
 \stackrel{\delta}{\longrightarrow} \cdots \]
is exact.
\end{lemma}
Let us recall a fundamental fact about double complexes.
\begin{prop} \label{prop.rowex}
If all the rows of an augmented double complex are exact, then the augmentation map
induces an isomorphism from the cohomology of the augmentation column to the
cohomology of the simple complex associated to the double complex.
\end{prop}
This fact is applied in showing:
\begin{prop} \label{prop.simp1}
The restriction map $r: \Omega^\bullet_\ms (B) \rightarrow
C^0 (\mathfrak{U}; \Omega^\bullet_\ms)$ induces an isomorphism
\[ r^\ast: H^\bullet (\Omega^\bullet_\ms (B)) \stackrel{\cong}{\longrightarrow}
H^\bullet (C^\bullet_\ms (\mathfrak{U}), D). \]
\end{prop}
The double complex $(C^\bullet (\pi^{-1} \mathfrak{U}; \Omega^\bullet),
\delta, d)$ given by
\[ C^p (\pi^{-1} \mathfrak{U}; \Omega^q) = \prod \Omega^q (\pi^{-1}
  U_{\alpha_0 \ldots \alpha_p}) \]
can be used to compute the cohomology of the total space $E$. The restriction map
\[ \overline{r}: \Omega^\bullet (E) \longrightarrow \prod_\alpha \Omega^\bullet
 (\pi^{-1} U_\alpha) = C^0 (\pi^{-1} \mathfrak{U}; \Omega^\bullet) \]
makes $C^\bullet (\pi^{-1} \mathfrak{U}; \Omega^\bullet)$ into an augmented
double complex. By the generalized Mayer-Vietoris sequence, Proposition 8.5 of
\cite{botttu}, the rows of this augmented double complex are exact. From
Proposition \ref{prop.rowex}, we thus deduce:
\begin{prop} \label{prop.simp2}
The restriction map $\overline{r}: \Omega^\bullet (E) \rightarrow
C^0 (\pi^{-1}\mathfrak{U}; \Omega^\bullet)$ induces an isomorphism
\[ \overline{r}^\ast: H^\bullet (E) = H^\bullet (\Omega^\bullet (E)) \stackrel{\cong}{\longrightarrow}
H^\bullet (C^\bullet (\pi^{-1}\mathfrak{U}), D), \]
where $(C^\bullet (\pi^{-1}\mathfrak{U}), D)$ is the simple complex of
$(C^\bullet (\pi^{-1}\mathfrak{U}; \Omega^\bullet), \delta, d)$.
\end{prop}

Regarding $\real^n \times F$ as a trivial fiber bundle over $\real^n$ with projection $\pi_1$, the
multiplicatively structured complex $\Omega^\bullet_\ms (\real^n)$ is defined as
\[ \Omega^\bullet_\ms (\real^n) = \{ \omega \in \Omega^\bullet (\real^n \times F) ~|~
 \omega = \sum_j \pi^\ast_1 \eta_j \wedge \pi^\ast_2 \gamma_j,~
 \eta_j \in \Omega^\bullet (\real^n),~ \gamma_j \in \Omega^\bullet (F) \}. \]
Let $s: \real^{n-1} \hookrightarrow \real \times \real^{n-1}  = \real^n$ be the standard inclusion
$s(u) = (0,u),$ $u\in \real^{n-1}.$ Let $q: \real^n = \real \times \real^{n-1} \rightarrow \real^{n-1}$
be the standard projection $q(t,u) =u,$ so that
$qs = \id_{\real^{n-1}}.$
Set
\[ S = s\times \id_F: \real^{n-1} \times F \hookrightarrow \real^n \times F,~ 
 Q = q\times \id_F: \real^{n} \times F \rightarrow \real^{n-1} \times F \]
so that $QS = \id_{\real^{n-1} \times F}.$
The induced map
$S^\ast: \Omega^\bullet (\real^n \times F) \rightarrow \Omega^\bullet (\real^{n-1} \times F)$
restricts to a map
\[ S^\ast: \Omega^\bullet_\ms (\real^n) \rightarrow \Omega^\bullet_\ms (\real^{n-1}). \]
The induced
map $Q^\ast: \Omega^\bullet (\real^{n-1} \times F) \rightarrow \Omega^\bullet (\real^n \times F)$
restricts to a map
\[ Q^\ast: \Omega^\bullet_\ms (\real^{n-1}) \rightarrow \Omega^\bullet_\ms (\real^{n}). \]

\begin{prop} \label{prop.rnfrnm1f}
The maps
\[ \xymatrix{\Omega^\bullet_{\ms}(\real^n) & \Omega^\bullet_\ms (\real^{n-1}) 
\ar@<1ex>[l]^{Q^\ast}  \ar@<1ex>[l];[]^{S^\ast}
} \]
are chain homotopy inverses of each other and thus induce mutually inverse
isomorphisms
\[ \xymatrix{H^\bullet (\Omega^\bullet_{\ms}(\real^n)) & H^\bullet (\Omega^\bullet_\ms (\real^{n-1})) 
\ar@<1ex>[l]^{Q^\ast}  \ar@<1ex>[l];[]^{S^\ast}
} \]
on cohomology.
\end{prop}

Let $S_0: F = \{ 0 \} \times F \hookrightarrow \real^n \times F$ be the inclusion
at $0$, inducing a map
$S^\ast_0: \Omega^\bullet_\ms (\real^n) \to \Omega^\bullet (F).$
The map $\pi^\ast_2: \Omega^\bullet (F) \rightarrow \Omega^\bullet
(\real^n \times F)$ restricts to a map
$\pi^\ast_2: \Omega^\bullet (F) \longrightarrow \Omega^\bullet_\ms
(\real^n).$
An induction on $n$ using Proposition \ref{prop.rnfrnm1f} shows:
\begin{prop} \label{prop.rntfharmf} (Poincar\'e Lemma for multiplicatively structured forms.)
The maps
\[ \xymatrix{\Omega^\bullet_{\ms}(\real^n) & \Omega^\bullet (F) 
\ar@<1ex>[l]^{\pi^\ast_2}  \ar@<1ex>[l];[]^{S^\ast_0}
} \]
are chain homotopy inverses of each other and thus induce mutually inverse
isomorphisms
\[ \xymatrix{H^\bullet (\Omega^\bullet_{\ms}(\real^n)) &  H^\bullet (F) 
\ar@<1ex>[l]^>>>>>{\pi^\ast_2}  \ar@<1ex>[l];[]^>>>>{S^\ast_0}
} \]
on cohomology.
\end{prop}
Using the classical Poincar\'e lemma, this Proposition readily implies:
\begin{prop} \label{prop.fhrnrntfiso}
The inclusion $\Omega^\bullet_\ms (\real^n)\subset \Omega^\bullet
(\real^n \times F)$ induces an isomorphism
\[ H^\bullet (\Omega^\bullet_\ms (\real^n)) \cong H^\bullet
(\real^n \times F) \]
on cohomology.
\end{prop}
\begin{prop} \label{prop.lociso}
For any $U_{\alpha_0 \ldots \alpha_p},$ the inclusion
$\Omega^\bullet_\ms (U_{\alpha_0 \ldots \alpha_p}) \hookrightarrow
 \Omega^\bullet (\pi^{-1} U_{\alpha_0 \ldots \alpha_p})$
induces an isomorphism on cohomology (with respect to the de Rham differential $d$).
\end{prop}

Since $d$ and $\delta$ on $C^\bullet (\mathfrak{U}; \Omega^\bullet_\ms)$
were obtained by restricting $d$ and $\delta$ on $C^\bullet
(\pi^{-1} \mathfrak{U}; \Omega^\bullet),$ the natural inclusion
$C^\bullet (\mathfrak{U}; \Omega^\bullet_\ms) \hookrightarrow
  C^\bullet (\pi^{-1} \mathfrak{U}; \Omega^\bullet)$
is a morphism of double complexes. It induces an isomorphism on vertical (i.e. $d$-)
cohomology by Proposition \ref{prop.lociso}. Whenever a morphism of
double complexes induces an isomorphism on vertical cohomology, then
it also induces an isomorphism on the $D$-cohomology of the respective
simple complexes. Consequently, using Propositions \ref{prop.simp1} and \ref{prop.simp2},
one gets:
\begin{thm} \label{thm.fhcomputescohtotspace}
The inclusion $\Omega^\bullet_\ms (B) \hookrightarrow \Omega^\bullet (E)$
induces an isomorphism
\[ H^\bullet (\Omega^\bullet_\ms (B)) \stackrel{\cong}{\longrightarrow}
H^\bullet (E) \]
on cohomology.
\end{thm}

\section{Fiberwise Truncation}
\label{sec.fibtrunc}

Before we discuss fiberwise (co)truncation, we must first discuss (co)truncation over a point.
Again, we refer to \cite{banagl-derhamintspace} for complete proofs of the facts cited in
this section.
We shall use the Riemannian metric on $F$ to define truncation
$\tau_{<k}$ and cotruncation $\tau_{\geq k}$ of the complex $\Omega^\bullet (F)$.
The bilinear form
\[ \begin{array}{rcl}
(\cdot, \cdot): \Omega^r (F) \times \Omega^r (F) & \longrightarrow & \real, \\
(\omega, \eta) & \mapsto & \int_F \omega \wedge *\eta,
\end{array} \]
where $*$ is the Hodge star, is symmetric and positive definite, thus defines an inner
product on $\Omega^\bullet (F)$.  The codifferential
\[ d^\ast = (-1)^{m(r+1)+1} *d*: \Omega^r (F) \longrightarrow \Omega^{r-1} (F) \]
is the adjoint of the differential $d$, $(d\omega, \eta) = (\omega, d^\ast \eta).$
The classical Hodge decomposition theorem provides orthogonal splittings
\begin{eqnarray*}
\Omega^r (F) & = & \im d^\ast \oplus \Harm^r (F) \oplus \im d, \\
\ker d & = & \Harm^r (F) \oplus \im d, \\
\ker d^\ast & = & \im d^\ast \oplus \Harm^r (F),
\end{eqnarray*}
where $\Harm^r (F) = \ker d \cap \ker d^\ast$ are the closed and coclosed, i.e. harmonic,
forms on $F$. In particular,
\[ \Omega^r (F) = \im d^\ast \oplus \ker d = \ker d^\ast \oplus \im d. \]
Let $k$ be a nonnegative integer.
\begin{defn} \label{def.truncoverpoint}
The \emph{truncation} $\tau_{<k} \Omega^\bullet (F)$ of $\Omega^\bullet (F)$ is the
complex
\[ \tau_{<k} \Omega^\bullet (F) = \cdots \longrightarrow \Omega^{k-2}(F)
\longrightarrow \Omega^{k-1}(F) \stackrel{d^{k-1}}{\longrightarrow} \im d^{k-1}
 \longrightarrow 0 \longrightarrow 0 \longrightarrow \cdots, \]
where $\im d^{k-1} \subset \Omega^k (F)$ is placed in degree $k$.
\end{defn}
The inclusion $\tau_{<k} \Omega^\bullet (F) \subset \Omega^\bullet (F)$ is a 
morphism of complexes.
The induced map on cohomology, $H^r (\tau_{<k} \Omega^\bullet F)\to
H^r (F),$ is an isomorphism for $r<k,$ while $H^r (\tau_{<k} \Omega^\bullet F)=0$ for
$r\geq k$.
\begin{defn} \label{def.cotruncoverpoint}
The \emph{cotruncation} $\tau_{\geq k} \Omega^\bullet (F)$ of $\Omega^\bullet (F)$ is the
complex
\[ \tau_{\geq k} \Omega^\bullet (F) = \cdots \longrightarrow 0
\longrightarrow 0 \longrightarrow \ker d^\ast
 \stackrel{d^k|}{\longrightarrow} \Omega^{k+1}(F) 
  \stackrel{d^{k+1}}{\longrightarrow} \Omega^{k+2}(F) \longrightarrow \cdots, \]
where $\ker d^\ast \subset \Omega^k (F)$ is placed in degree $k$.
\end{defn}
The inclusion $\tau_{\geq k} \Omega^\bullet (F) \subset \Omega^\bullet (F)$ is a
morphism of complexes. By construction, we have $H^r (\tau_{\geq k} \Omega^\bullet F)=0$
for $r<k$ and the inclusion 
$\tau_{\geq k} \Omega^\bullet (F) \hookrightarrow \Omega^\bullet (F)$
induces an isomorphism $H^r (\tau_{\geq k} \Omega^\bullet F)
\stackrel{\cong}{\longrightarrow} H^r (F)$ in the range $r\geq k$.
A key advantage of cotruncation over truncation is that $\tau_{\geq k} \Omega^\bullet F$
is a subalgebra of $(\Omega^\bullet F,\wedge)$, whereas $\tau_{<k} \Omega^\bullet F$ is not.
\begin{prop} \label{prop.indepriemmetric}
The isomorphism type of $\tau_{\geq k} \Omega^\bullet F$ in the category of cochain complexes
is independent of the Riemannian metric on $F$.
\end{prop}
\begin{lemma} \label{lem.isometrypreservcotrunc}
Let $f:F\to F$ be a smooth self-map. \\
\noindent (1) $f$ induces an endomorphism $f^\ast$ of $\tau_{<k} \Omega^\bullet F$. \\
\noindent (2) If $f$ is an isometry, then $f$ induces an automorphism $f^\ast$
 of $\tau_{\geq k}  \Omega^\bullet F$.
\end{lemma}

We shall next define the \emph{fiberwise truncation}
$\ft_{<t} \oms^\bullet (\real^n) \subset \oms^\bullet (\real^n)$ and the 
\emph{fiberwise cotruncation} $\ft_{\geq t} \oms^\bullet (\real^n) \subset
\oms (\real^n),$ depending on an integer $t$. 
Set
\[ \ft_{<t} \oms^\bullet (\real^n) = \{ \omega \in \Omega^\bullet (\real^n \times F) ~|~
  \omega = \sum_j \pi^\ast_1 \eta_j \wedge \pi^\ast_2 \gamma_j, \]
\[ \hspace{4cm} \eta_j \in \Omega^\bullet (\real^n),~ 
\gamma_j \in \tau_{<t} \Omega^\bullet (F) \}. \]
The complex $\ft_{<t} \oms^\bullet (\real^n)$ is a
subcomplex of $\oms^\bullet (\real^n).$ Define
\[ \ft_{\geq t} \oms^\bullet (\real^n) = \{ \omega \in \Omega^\bullet (\real^n \times F) ~|~
  \omega = \sum_j \pi^\ast_1 \eta_j \wedge \pi^\ast_2 \gamma_j, \]
\[ \hspace{4cm} \eta_j \in \Omega^\bullet (\real^n),~ 
\gamma_j \in \tau_{\geq t} \Omega^\bullet (F) \}. \]
Again, this is a subcomplex of $\oms^\bullet (\real^n)$. 
Similarly, a subcomplex
$\ft_{<t} \Omega^\bullet_{\ms} (U_{\alpha_0 \ldots \alpha_p}) \subset
\Omega^\bullet_{\ms} (U_{\alpha_0 \ldots \alpha_p})$ of fiberwise truncated
multiplicatively structured forms on $\pi^{-1}(U_{\alpha_0 \ldots \alpha_p})$ is given
by requiring every $\gamma_j$ to lie in $\tau_{<t} \Omega^\bullet (F).$ This is well-defined
by the transformation law (\ref{equ.transformlaw}) together with Lemma \ref{lem.isometrypreservcotrunc} (1).
A subcomplex
$\ft_{\geq t} \Omega^\bullet_{\ms} (U_{\alpha_0 \ldots \alpha_p}) \subset
\Omega^\bullet_{\ms} (U_{\alpha_0 \ldots \alpha_p})$ of fiberwise cotruncated
multiplicatively structured forms on $\pi^{-1}(U_{\alpha_0 \ldots \alpha_p})$ is given
by requiring every $\gamma_j$ to lie in $\tau_{\geq t} \Omega^\bullet (F).$ This is well-defined
by the transformation law and Lemma \ref{lem.isometrypreservcotrunc} (2). (At this point, it is
used that the transition functions of the bundle are isometries.)

Let
$S: \real^{n-1} \times F \hookrightarrow \real^n \times F,~
  Q: \real^n \times F \longrightarrow \real^{n-1} \times F$
be as in Section \ref{sec.fibharmflat}.
The induced map $S^\ast: \oms^\bullet (\real^n)\rightarrow \oms^\bullet (\real^{n-1})$ restricts
to a map
\[ S^\ast: \ft_{<t} \oms^\bullet (\real^n)\longrightarrow 
   \ft_{<t} \oms^\bullet (\real^{n-1}). \]
The induced map
$Q^\ast: \oms^\bullet (\real^{n-1})\rightarrow \oms^\bullet (\real^n)$ restricts
to a map
\[ Q^\ast: \ft_{<t} \oms^\bullet (\real^{n-1})\longrightarrow 
   \ft_{<t} \oms^\bullet (\real^n). \]

\begin{lemma} \label{lem.821}
The maps
\[ \xymatrix{\ft_{<t} \oms^\bullet (\real^n) & \ft_{<t} \oms^\bullet (\real^{n-1}) 
\ar@<1ex>[l]^{Q^\ast}  \ar@<1ex>[l];[]^{S^\ast}
} \]
are chain homotopy inverses of each other and thus induce mutually inverse
isomorphisms
\[ \xymatrix{H^\bullet (\ft_{<t} \oms^\bullet (\real^n)) & 
  H^\bullet (\ft_{<t} \oms^\bullet (\real^{n-1})) 
\ar@<1ex>[l]^{Q^\ast}  \ar@<1ex>[l];[]^{S^\ast}
} \]
on cohomology.
\end{lemma}

As in Section \ref{sec.fibharmflat}, let $S_0: F = \{ 0 \} \times F \hookrightarrow
\real^n \times F$ be the inclusion at $0$. The induced map
$S^\ast_0: \oms^\bullet (\real^n) \rightarrow \Omega^\bullet (F)$ restricts to a map
$S^\ast_0: \ft_{<t} \oms^\bullet (\real^n) \longrightarrow \tau_{<t} \Omega^\bullet (F).$
The map $\pi^\ast_2: \Omega^\bullet (F) \rightarrow \oms^\bullet (\real^n)$ restricts
to a map
$\pi^\ast_2: \tau_{<t} \Omega^\bullet (F) \rightarrow \ft_{<t}\oms^\bullet (\real^n).$
An induction on $n$ using Lemma \ref{lem.821} shows:
\begin{lemma}(Poincar\'e Lemma for truncated forms.) \label{lem.822}
The maps
\[ \xymatrix{\ft_{<t} \oms^\bullet (\real^n) & \tau_{<t} \Omega^\bullet (F) 
\ar@<1ex>[l]^{\pi^\ast_2}  \ar@<1ex>[l];[]^{S^\ast_0}
} \]
are chain homotopy inverses of each other and thus induce mutually inverse
isomorphisms
\[ \xymatrix{H^r (\ft_{<t} \oms^\bullet (\real^n)) & 
  H^r (\tau_{<t} \Omega^\bullet (F)) \cong 
\mbox{$\begin{cases} H^r (F),& r<t \\ 0, & r\geq t. \end{cases}$} 
\ar@<1ex>[l]^<<<{\pi^\ast_2}  \ar@<1ex>[l];[]^<<<<{S^\ast_0}
} \]
on cohomology.
\end{lemma}
An analogous argument, replacing $\tau_{<t} \Omega^\bullet (F)$ by
$\tau_{\geq t} \Omega^\bullet (F),$ proves a version for fiberwise cotruncation:
\begin{lemma}(Poincar\'e Lemma for cotruncated forms.) \label{lem.822cotrunc}
The maps
\[ \xymatrix{\ft_{\geq t} \oms^\bullet (\real^n) & \tau_{\geq t} \Omega^\bullet (F) 
\ar@<1ex>[l]^{\pi^\ast_2}  \ar@<1ex>[l];[]^{S^\ast_0}
} \]
are chain homotopy inverses of each other and thus induce mutually inverse
isomorphisms
\[ \xymatrix{H^r (\ft_{\geq t} \oms^\bullet (\real^n)) & 
  H^r (\tau_{\geq t} \Omega^\bullet (F)) \cong 
\mbox{$\begin{cases} H^r (F),& r\geq t \\ 0, & r<t. \end{cases}$} 
\ar@<1ex>[l]^<<<{\pi^\ast_2}  \ar@<1ex>[l];[]^<<<<{S^\ast_0}
} \]
on cohomology.
\end{lemma}

\section{\v{C}ech Presheaves}

Let $M$ be a smooth manifold and $\mathfrak{V} = \{ V_\alpha \}$ be a good open
cover of $M$. The cover $\mathfrak{V}$ gives rise to a category $C(\mathfrak{V})$,
whose objects are all finite intersections $V_{\alpha_0 \ldots \alpha_p}$ of open sets
$V_\alpha$ in $\mathfrak{V}$ and an initial object $\varnothing$, the empty set.
The morphisms are inclusions.
A \emph{\v{C}ech presheaf $\der$ on $\mathfrak{V}$} is a contravariant
functor $\der: C(\mathfrak{V}) \to \real \text{-}\operatorname{MOD}$ into the
category $\real \text{-}\operatorname{MOD}$ of real vector spaces and linear maps,
such that $\der (\varnothing)=0$.
A \emph{homomorphism} $\der \to \bG$ of \v{C}ech presheaves on $\mathfrak{V}$ is a
natural transformation from $\der$ to $\bG$.
The homomorphism is an isomorphism if
$\der (V_{\alpha_0 \ldots \alpha_p}) \to \bG (V_{\alpha_0 \ldots \alpha_p})$ is an
isomorphism for every object $V_{\alpha_0 \ldots \alpha_p}$ in $C(\mathfrak{V})$.
Let $H$ be a real vector space. The presheaf $\der$ is said to be \emph{locally constant
with group $H$}, if all $\der (V_{\alpha_0 \ldots \alpha_p})$, for
$V_{\alpha_0 \ldots \alpha_p} \not= \varnothing,$ are isomorphic to $H$
and all linear maps $\der (V_{\alpha_0 \ldots \alpha_p}) \to \der (V_{\beta_0 \ldots \beta_q})$
for nonempty inclusions $V_{\beta_0 \ldots \beta_q} \subset V_{\alpha_0 \ldots \alpha_p}$ are
isomorphisms. A \v{C}ech presheaf $\der$ on $\mathfrak{V}$ possesses 
$p$-cochains
\[ C^p (\mathfrak{V}; \der) = \prod_{\alpha_0, \cdots, \alpha_p} 
  \der (V_{\alpha_0 \ldots \alpha_p}) \]
and a \v{C}ech differential $\delta: C^p (\mathfrak{V}; \der) \to
C^{p+1} (\mathfrak{V}; \der)$ making $C^\bullet (\mathfrak{V}; \der)$ into a complex.
Its cohomology $H^p (\mathfrak{V}; \der)$ is the \emph{\v{C}ech cohomology
of the cover $\mathfrak{V}$ with values in $\der$.} 
An isomorphism
$\der \stackrel{\cong}{\longrightarrow} \bG$ of two presheaves on $\mathfrak{V}$
induces an isomorphism of cohomology groups
$H^\bullet (\mathfrak{V}; \der) \stackrel{\cong}{\longrightarrow}
H^\bullet (\mathfrak{V}; \bG).$ 
Let $\bL$ be a locally constant sheaf on $M$. Then $\bL$ defines in particular a
locally constant presheaf on $\mathfrak{V}$ (since $\mathfrak{V}$ is good,
$\bL|_V$ is constant for every $V\in Ob C(\mathfrak{V})$) so that
$H^\bullet (\mathfrak{V};\bL)$ is defined. Then, as
$H^q (V_{\alpha_0 \ldots \alpha_p};\bL)=0$ for $q>0,$ there is a canonical
isomorphism 
$H^\bullet_{\operatorname{Sh}} (M;\bL)\cong H^\bullet (\mathfrak{V};\bL)$
according to \cite[Thm. III.4.13]{bredon}, where
$H^\bullet_{\operatorname{Sh}} (M;\bL)$ denotes sheaf cohomology.
In particular, the \v{C}ech cohomology groups of $\mathfrak{V}$ with values
in $\bL$ are independent of the good cover used to define them. By
\cite[Cor. III.4.12]{bredon}, these groups are furthermore canonically isomorphic
to the \v{C}ech cohomology of $M$ with coefficients in $\bL$, 
$\check{H}^\bullet (M;\bL)$. 
If we view $\bL$ as a local coefficient system on $M$, then the singular
cohomology $H^\bullet (M;\bL)$ is defined and a canonical isomorphism
$H^\bullet_{\operatorname{Sh}} (M;\bL) \cong H^\bullet (M;\bL)$ is 
provided by \cite[Thm. III.1.1]{bredon}. \\

Let us return to the good cover $\fU$ on our base space $B$. We shall define
three \v{C}ech presheaves on $\fU$. Define $\der^q (F)$ by
\[ \der^q (F)(U) = H^q (\pi^{-1} U),~ U\in Ob C(\fU). \]
The structural morphisms associated to inclusions are given by restriction of
forms. Since all nonempty objects $U$ in $C(\fU)$ are diffeomorphic to $\real^n$, and
the bundle $\pi: E\to B$ trivializes over every such $U$, the classical 
Poincar\'e lemma implies that $\der^q (F)$ is a locally constant presheaf with
group $H^q (F),$ the de Rham cohomology of the fiber. Define the presheaf
$\der^q_\ms (F)$ by
\[ \der^q_\ms (F)(U) = H^q (\Omega^\bullet_\ms (U)),~ U\in Ob C(\fU). \]
According to Proposition \ref{prop.lociso}, the inclusion
$\Omega^\bullet_\ms (U) \subset \Omega^\bullet (\pi^{-1} U)$ induces
an isomorphism $\der^q_\ms (F)(U) \stackrel{\cong}{\longrightarrow}
\der^q (F)(U)$ for every nonempty $U\in Ob C(\fU)$. If
$V\in Ob C(\fU)$ is an open set with $V\subset U$, then the commutative
square
\[ \xymatrix{
\Omega^\bullet_\ms (U) \ar@{^{(}->}[r] \ar[d]_{\operatorname{restr}}&
 \Omega^\bullet (\pi^{-1} U) \ar[d]^{\operatorname{restr}} \\
\Omega^\bullet_\ms (V) \ar@{^{(}->}[r] &
 \Omega^\bullet (\pi^{-1} V) 
} \]
induces a commutative square
\[ \xymatrix{
\der^q_\ms (F)(U) \ar[r]^{\cong} \ar[d]_{\operatorname{restr}} &
  \der^q (F)(U) \ar[d]^{\operatorname{restr}} \\
\der^q_\ms (F)(V) \ar[r]^{\cong} &
  \der^q (F)(V). 
} \]
Thus the inclusion of multiplicatively structured forms induces an isomorphism
\[ \der^q_\ms (F) \stackrel{\cong}{\longrightarrow} \der^q (F) \]
of presheaves for every $q$. In particular, $\der^q_\ms (F)$ is also
locally constant with group $H^q (F)$. The isomorphism induces
furthermore an isomorphism
\begin{equation} \label{equ.nabla}
H^p (\fU; \der^q_\ms (F)) \stackrel{\cong}{\longrightarrow}
H^p (\fU; \der^q (F))
\end{equation}
of \v{C}ech cohomology groups. Define the presheaf
$\der^q_{\geq t} (F)$ by
\[ \der^q_{\geq t} (F)(U) = H^q (\ft_{\geq t} \Omega^\bullet_\ms (U)),~ 
  U\in Ob C(\fU), U\not= \varnothing. \]
Since $U \not= \varnothing$ is diffeomorphic to $\real^n$ and the bundle $\pi: E\to B$
trivializes over $U$, the Poincar\'e Lemma \ref{lem.822cotrunc} for
cotruncated forms implies that the restriction $S^\ast_0$ of a form
to the fiber over the origin of $U\cong \real^n$ induces an isomorphism
\[ \xymatrix{
\der^q_{\geq t} (F)(U) \ar[r]^{\cong}_{S^\ast_0} & H^q (F) 
} \]
for $q\geq t,$ whereas $\der^q_{\geq t} (F)(U)=0$ for $q<t$.
If $V \in Ob C(\fU)$ is a nonempty open set with $V\subset U$, then the
commutative diagram
\[ \xymatrix{
\ft_{\geq t} \Omega^\bullet_\ms (U) \ar[r]^{S^\ast_0} 
 \ar[d]_{\operatorname{restr}} &
  \tau_{\geq t} \Omega^\bullet (F) \\
\ft_{\geq t} \Omega^\bullet_\ms (V) \ar[ru]_{S^\ast_0} &
} \]
(assuming, without loss of generality, that the origin of $U$ lies in $V$)
induces a commutative diagram
\[ \xymatrix{
\der^q_{\geq t} (F)(U) \ar[r]^{\cong}_{S^\ast_0} 
 \ar[d]_{\operatorname{restr}} &
  H^q (F) \\
\der^q_{\geq t} (F)(V) \ar[ru]^{\cong}_{S^\ast_0} &
} \]
for $q\geq t$. It follows that the restriction induces an isomorphism
$\der^q_{\geq t} (F)(U) \stackrel{\cong}{\longrightarrow}
 \der^q_{\geq t} (F)(V).$ Therefore, when $q\geq t,$ 
$\der^q_{\geq t} (F)$ is a locally constant presheaf with group
$H^q (F).$ Moreover, the commutative diagram
\[ \xymatrix{
\ft_{\geq t} \Omega^\bullet_\ms (U) \ar[r]^{S^\ast_0} 
 \ar@{^{(}->}[d] &
  \Omega^\bullet (F) \\
\Omega^\bullet_\ms (U) \ar[ru]_{S^\ast_0} &
} \]
induces, for $q\geq t,$ a commutative diagram
\[ \xymatrix{
\der^q_{\geq t} (F)(U) \ar[r]^{\cong}_{S^\ast_0} 
 \ar[d] &
  H^q (F) \\
\der^q_\ms (F)(U), \ar[ru]^{\cong}_{S^\ast_0} &
} \]
using Proposition \ref{prop.rntfharmf}. Thus
$\der^q_{\geq t} (F)(U) \to \der^q_\ms (F)(U)$ is an
isomorphism for $q\geq t.$ Since it commutes with restriction to
smaller open sets, we obtain the following result.

\begin{lemma} \label{lem.3.1}
For $q<t,$ the presheaf $\der^q_{\geq t} (F)$ is trivial,
$\der^q_{\geq t} (F)=0.$ For $q\geq t,$ $\der^q_{\geq t} (F)$
is locally constant with group $H^q (F)$ and the inclusion
$\ft_{\geq t} \Omega^\bullet_\ms (-)\subset \Omega^\bullet_\ms (-)$
induces an isomorphism
\[ \der^q_{\geq t} (F) \stackrel{\cong}{\longrightarrow}
 \der^q_\ms (F) \]
of presheaves.
\end{lemma}

\section{The Spectral Sequence of a Flat, Isometrically
 Structured Bundle}

Let $(K,\delta, d)$ be the double complex
\[ K^{p,q} = C^p (\pi^{-1} \fU; \Omega^q) =
 \prod_{\alpha_0, \cdots, \alpha_p}
 \Omega^q (\pi^{-1} U_{\alpha_0 \ldots \alpha_p}) \]
defined in Section \ref{sec.fibharmflat}. The \emph{spectral
sequence of the fiber bundle} $\pi: E\to B$ is the spectral
sequence $E(K) = \{ E_r, d_r \}$ of $K$. Its $E_1$-term is
\[ E^{p,q}_1 = H^{p,q}_d (K) = 
 \prod_{\alpha_0, \cdots, \alpha_p}
  H^q (\pi^{-1} U_{\alpha_0 \ldots \alpha_p}) =
 C^p (\fU; \der^q (F)). \]
Since $d_1 = \delta$ on $E_1$, the $E_2$-term is
\[ E^{p,q}_2 = H^p (\fU; \der^q (F)) \cong H^p (B; \der^q (F)), \]
the (singular) cohomology of $B$ with values in the local
coefficient system $\der^q (F)$.

\begin{thm} \label{thm.main1}
Let $F$ be a closed, oriented, Riemannian manifold and
$\pi: E\to B$ a flat, smooth fiber bundle over the closed, smooth
base manifold $B$ with fiber $F$ and structure group the
isometries of $F$. Then the spectral sequence with real coefficients of
$\pi: E\to B$ collapses at the $E_2$-term.
\end{thm}
\begin{proof}
Using multiplicatively structured forms, we first build a smaller model
$K_\ms$ of $K$. The spectral sequences of $K_\ms$ and $K$ will be
shown to be isomorphic (from the $E_2$-term on). In Section
\ref{sec.fibharmflat}, we introduced the multiplicatively structured
\v{C}ech - de Rham double complex
$K_\ms = (C^\bullet (\fU; \Omega^\bullet_\ms), \delta, d)$. In
bidegree $(p,q)$ it is given by
\[ K^{p,q}_\ms = C^p (\fU; \Omega^q_\ms) =
 \prod_{\alpha_0, \cdots, \alpha_p}
 \Omega^q_\ms (U_{\alpha_0 \ldots \alpha_p}). \]
The vertical cohomology of $K_\ms$ is
\[ H^{p,q}_d (K_\ms) = \prod_{\alpha_0, \cdots, \alpha_p}
 H^q (\Omega^\bullet_\ms (U_{\alpha_0 \ldots \alpha_p})) =
 C^p (\fU; \der^q_\ms (F)). \]
The core of the argument is the construction of a filtration of
$K_\ms$ by cotruncated double complexes $K_{\geq t} \subset
K_\ms$, where $t$ is an integer. The group in bidegree $(p,q)$ is
\[ K^{p,q}_{\geq t} = C^p (\fU; (\ft_{\geq t} \Omega^\bullet_\ms)^q) =
 \prod_{\alpha_0, \cdots, \alpha_p}
 (\ft_{\geq t} \Omega^\bullet_\ms (U_{\alpha_0 \ldots \alpha_p}))^q. \]
The vertical differential is given by the (restriction of the) de Rham
differential $d$, and the horizontal differential is given by the
\v{C}ech differential $\delta$. The vertical cohomology of $K_{\geq t}$
is
\[ H^{p,q}_d (K_{\geq t}) = \prod_{\alpha_0, \cdots, \alpha_p}
 H^q (\ft_{\geq t} \Omega^\bullet_\ms (U_{\alpha_0 \ldots \alpha_p})) =
 C^p (\fU; \der^q_{\geq t} (F)). \]
The double complex $K_\ms$ determines a spectral sequence
$E(K_\ms) = \{ E_{\ms, r}, d_{\ms, r} \};$ the 
double complex $K_{\geq t}$ determines a spectral sequence
$E(K_{\geq t}) = \{ E_{\geq t, r}, d_{\geq t, r} \},$
cf. \cite[Thm. 14.14, page 165]{botttu}. The inclusions of complexes
\[ \ft_{\geq t} \Omega^\bullet_\ms (U) \subset \Omega^\bullet_\ms (U)
 \subset \Omega^\bullet (\pi^{-1} U),~ U\in Ob C(\fU), \]
induce inclusions of double complexes
\[ K_{\geq t} \subset K_{\ms} \subset K. \]
A map of double complexes induces a morphism of the associated
spectral sequences. Thus the above inclusions induce morphisms
\[ E(K_{\geq t}) \stackrel{e}{\longrightarrow} E(K_\ms) \stackrel{f}{\longrightarrow}
 E(K). \]
Let us show that the differentials $d_{\ms, 2}$ vanish. This will then
provide the induction basis for an inductive proof that 
all $d_{\ms, r},$ $r\geq 2,$ vanish. The term $E_{\ms,1}$ is given by
\[ E^{p,q}_{\ms, 1} = H^{p,q}_d (K_\ms) = 
 C^p (\fU; \der^q_\ms (F)). \]
Since $d_{\ms, 1} = \delta,$ we have
\[ E^{p,q}_{\ms, 2} = H^p_\delta (\fU; \der^q_\ms (F)), \]
the \v{C}ech cohomology of $\fU$ with values in the presheaf
$\der^q_\ms (F)$. The term $E_{\geq t,1}$ is given by
\[ E^{p,q}_{\geq t, 1} = H^{p,q}_d (K_{\geq t}) = 
 C^p (\fU; \der^q_{\geq t} (F)). \]
Since $d_{\geq t, 1} = \delta,$ we have
\[ E^{p,q}_{\geq t, 2} = H^p_\delta (\fU; \der^q_{\geq t} (F)). \]
Set $t=q$. Since $e$ is a morphism of spectral sequences, we have
a commutative square
\[ \xymatrix{
E^{p,q}_{\geq t, 2} \ar[r]^{e^{p,q}_2} \ar[d]_{d^{p,q}_{\geq t,2}} &
 E^{p,q}_{\ms, 2} \ar[d]^{d^{p,q}_{\ms, 2}} \\
E^{p+2, q-1}_{\geq t, 2} \ar[r]^{e^{p+2, q-1}_2} &
 E^{p+2, q-1}_{\ms, 2}.
} \]
In view of the above identifications of $E_2$-terms, this can be rewritten as
\[ \xymatrix{
H^p (\fU; \der^q_{\geq t} (F)) \ar[r]^{e^{p,q}_2} \ar[d]_{d^{p,q}_{\geq t,2}} &
 H^p (\fU; \der^q_\ms (F)) \ar[d]^{d^{p,q}_{\ms, 2}} \\
 H^{p+2} (\fU; \der^{q-1}_{\geq t} (F)) \ar[r]^{e^{p+2, q-1}_2} &
 H^{p+2} (\fU; \der^{q-1}_\ms (F)) .
} \]
Our choice of $t$ together with Lemma \ref{lem.3.1} implies that
$e^{p,q}_2$ is an isomorphism. Therefore, we can express $d^{p,q}_{\ms, 2}$
as the composition
\begin{equation} \label{equ.twostars}
d^{p,q}_{\ms, 2} = e^{p+2, q-1}_2 \circ d^{p,q}_{\geq t, 2} \circ
  (e^{p,q}_2)^{-1}.
\end{equation}
Since $q-1 < t$, we have $\der^{q-1}_{\geq t} (F)=0$ by Lemma \ref{lem.3.1},
and thus $H^{p+2} (\fU; \der^{q-1}_{\geq t} (F))=0,$
$d^{p,q}_{\geq t, 2} =0$ and $e^{p+2, q-1}_2 =0.$ By
(\ref{equ.twostars}), $d^{p,q}_{\ms, 2}=0.$ \\

We shall next show that for arbitrary $t$, $d_{\geq t,2} =0.$
Given any bidegree $(p,q)$, there are two cases to consider: $q-1<t$ and
$q-1\geq t$. If $q-1<t,$ then we have $E^{p+2, q-1}_{\geq t, 2} =$
$H^{p+2} (\fU; \der^{q-1}_{\geq t} (F))=0,$ so that
$d^{p,q}_{\geq t,2} : E^{p,q}_{\geq t,2} \to E^{p+2, q-1}_{\geq t,2} =0$
is zero. If $q-1 \geq t,$ then $e^{p+2, q-1}_2$ is an isomorphism by Lemma
\ref{lem.3.1} and
\[ d^{p,q}_{\geq t,2} = (e^{p+2, q-1}_2)^{-1} \circ d^{p,q}_{\ms, 2} \circ
  e^{p,q}_2 =0, \]
since $d^{p,q}_{\ms, 2}=0.$ Thus $d_{\geq t, 2}=0$ for any $t$. \\

For $r\geq 2,$ let $P(r)$ denote the package of statements
\begin{itemize}
\item $d_{\ms, r}=0$ and $d_{\geq t, r}=0$ for all $t$, 
\item $E^{p,q}_{\ms, r} = H^p (\fU; \der^q_\ms (F))$ and
  $E^{p,q}_{\geq t, r} = H^p (\fU; \der^q_{\geq t} (F))$ for all $t$, and
\item $e_{r} = e_2$.
\end{itemize}
We have shown that $P(2)$ holds. We shall now show for $r\geq 3$ that if
$P(r-1)$ holds, then $P(r)$ holds. The vanishing of the differentials in the
$E_{r-1}$-terms implies that
\[ E^{p,q}_{\ms,r} = E^{p,q}_{\ms,r-1} = H^p (\fU; \der^q_\ms (F)) \]
and
\[ E^{p,q}_{\geq t,r} = E^{p,q}_{\geq t,r-1} = H^p (\fU; \der^q_{\geq t} (F)). \]
Furthermore, as $e$ is a morphism of spectral sequences, we have
\[ e^{p,q}_r = H^{p,q} (e_{r-1}) = e^{p,q}_{r-1} = e^{p,q}_2. \]
Hence the commutative square
\[ \xymatrix@C=40pt{
E^{p,q}_{\geq t, r} \ar[r]^{e^{p,q}_r} \ar[d]_{d^{p,q}_{\geq t,r}} &
 E^{p,q}_{\ms, r} \ar[d]^{d^{p,q}_{\ms, r}} \\
E^{p+r, q-r+1}_{\geq t, r} \ar[r]^{e^{p+r, q-r+1}_r} &
 E^{p+r, q-r+1}_{\ms, r}
} \]
can be rewritten as
\[ \xymatrix@C=40pt{
H^p (\fU; \der^q_{\geq t} (F)) \ar[r]^{e^{p,q}_2} \ar[d]_{d^{p,q}_{\geq t,r}} &
 H^p (\fU; \der^q_\ms (F)) \ar[d]^{d^{p,q}_{\ms, r}} \\
 H^{p+r} (\fU; \der^{q-r+1}_{\geq t} (F)) \ar[r]^{e^{p+r, q-r+1}_2} &
 H^{p+r} (\fU; \der^{q-r+1}_\ms (F)).
} \]
Again take $t=q$. Then $e^{p,q}_2$ is an isomorphism and the factorization
\[
d^{p,q}_{\ms, r} = e^{p+r, q-r+1}_2 \circ d^{p,q}_{\geq t, r} \circ
  (e^{p,q}_2)^{-1}
\]
shows that $d^{p,q}_{\ms, r}=0$ because
$H^{p+r} (\fU; \der^{q-r+1}_{\geq t} (F))=0$ by $q-r+1<t$ ($r\geq 3$).
For arbitrary $t$, $d_{\geq t, r} =0.$ For if $q-r+1<t,$ then
$E^{p+r, q-r+1}_{\geq t, r} = H^{p+r} (\fU; \der^{q-r+1}_{\geq t} (F))=0$
so that $d^{p,q}_{\geq t, r}=0,$ while for $q-r+1\geq t,$ the map
$e^{p+r, q-r+1}_2$ is an isomorphism and $d^{p,q}_{\geq t, r}=0$ follows
from the factorization
$d^{p,q}_{\geq t, r} = (e^{p+r, q-r+1}_2)^{-1} \circ d^{p,q}_{\ms, r}
\circ e^{p,q}_2$ and the fact that $d_{\ms, r}=0$. This induction shows
that $d_{\ms, r}=0$ for all $r\geq 2$. We conclude that $E(K_\ms)$ collapses
at the $E_2$-term. This will now be used to prove that $E(K)$ collapses
at $E_2$. \\

Since $f$ is a morphism of spectral sequences, we have for every
$(p,q)$ a commutative square
\[ \xymatrix@C=40pt{
E^{p,q}_{\ms, 2} \ar[r]^{f^{p,q}_2} \ar[d]_{0=d^{p,q}_{\ms,2}} &
 E^{p,q}_{2} \ar[d]^{d^{p,q}_{2}} \\
E^{p+2, q-1}_{\ms, 2} \ar[r]^{f^{p+2, q-1}_2} &
 E^{p+2, q-1}_{2},
} \]
which can be rewritten as
\[ \xymatrix@C=40pt{
H^p (\fU; \der^q_{\ms} (F)) \ar[r]^{f^{p,q}_2} \ar[d]_{0} &
 H^p (\fU; \der^q (F)) \ar[d]^{d^{p,q}_{2}} \\
 H^{p+2} (\fU; \der^{q-1}_{\ms} (F)) \ar[r]^{f^{p+2, q-1}_2} &
 H^{p+2} (\fU; \der^{q-1} (F)) .
} \]
By (\ref{equ.nabla}), $f^{p,q}_2$ is an isomorphism for all 
$(p,q)$. Thus $d_2 =0$ on $E_2$. The fact that $f_2$ is an isomorphism
implies that $f_r$ is an isomorphism for all $r\geq 2$. This shows,
since $d_{\ms, r} =0$ for all $r\geq 2,$ that $d^{p,q}_r =0$ for all
$r\geq 2$.  Consequently, $E(K)$ collapses at the $E_2$-term, as
was to be shown.
\end{proof}

\section{Nonisometrically Structured Flat Bundles}
\label{sec.nonisometric}

We construct an example of a flat, smooth circle bundle whose Leray-Serre spectral
sequence with real coefficients does not collapse at the $E_2$-term.
The example shows, then, that in Theorem \ref{thm.main1} one cannot delete
without substitution the requirement that the structure group act isometrically.
The example is based on constructions of J. Milnor, cf. \cite{milnorconncurvzero},
\cite{milnorstasheff}. \\

Let $B$ be a closed Riemann surface of genus $2$. Its universal cover
$\widetilde{B}$ is conformally diffeomorphic to the complex upper half plane
$H = \{ z\in \cplx ~|~ \operatorname{Im} z>0 \}.$
The fundamental group $\pi_1 (B)$ acts on $H$ biholomorphically as the
deck transformations. The group of biholomorphic automorphisms of $H$ is
$\operatorname{PSL}(2,\real)$, acting as M\"obius transformations
\[ z\mapsto \frac{az+b}{cz+d},~ a,b,c,d\in \real,~ ad-bc=1. \]
This yields a faithful representation $\pi_1 B \to \operatorname{PSL}(2,\real).$
The operation of $\operatorname{PSL}(2,\real)$ on $H$ extends naturally to the
closure $\overline{H} = H\cup \real \cup \{ \infty \}$ of $H$ in $\cplx \cup
\{ \infty \}$. In particular, $\operatorname{PSL}(2,\real)$ acts on the boundary
circle $\partial \overline{H} = \real \cup \{ \infty \}$ and we can form the
flat circle bundle $\gamma$ with projection
$\pi: E = \widetilde{B} \times_{\pi_1 B} (\real \cup \{ \infty \}) \to B$ and
structure group $\operatorname{PSL}(2,\real)$. Recall that any orientable,
possibly nonlinear, sphere bundle $\xi$ with structure group 
$\operatorname{Diff}(S^{n-1})$ has a real Euler class
$e(\xi) \in H^n (M;\real),$ where $M$ is the base manifold. Since our bundle
$\gamma$ can be identified (though not linearly) with the tangent circle
bundle $S(TB)$ of $B$ (see \cite{milnorstasheff}), the Euler number
$\langle e(\gamma), [B] \rangle$ of $\gamma$ is the Euler characteristic
$\chi (B) =-2$. As the Euler class is transgressive and not zero, the differential
$d_2: E^{0,1}_2 \to E^{2,0}_2$ is nontrivial and the spectral sequence of
$\gamma$ does not collapse at $E_2$. In more detail, let
$\sigma \in E^{0,1}_1 = C^0 (\mathfrak{U}; \der^1 (S^1))$ be the element
corresponding to the usual angular forms on $E$
(\cite[Remark 14.20]{botttu}). Since $\gamma$ is orientable,
$d_1 (\sigma) = \delta (\sigma)=0$ and $\sigma$ determines a class
$[\sigma ]\in E^{0,1}_2 = H^0 (B;\der^1 (S^1)).$
We have $e(\gamma)=d_2 [\sigma]$ for the transgression
\[ d_2: E^{0,1}_2 = H^0 (B;\der^1 (S^1))\longrightarrow
  E^{2,0}_2 = H^2 (B;\der^0 (S^1))= H^2 (B;\real). \]
It follows from Theorem \ref{thm.main1} that there is no Riemannian metric on
$S^1 = \real \cup \{ \infty \}$ such that the action of $\pi_1 (B)$ on
$\real \cup \{ \infty \}$ is isometric. This statement will now be
affirmed directly, without appealing to the theorem.

\begin{prop} \label{prop.nonisometric}
There is no Riemannian metric on $\real \cup \{ \infty \}$ such that the
Fuchsian group given by the image of the representation
$\pi_1 B \to \operatorname{PSL}(2,\real)$ acts isometrically on
$\real \cup \{ \infty \}$.
\end{prop}
\begin{proof}
By contradiction; suppose there were such a metric.
The image $\Gamma \subset \operatorname{PSL}(2,\real)$ of the
faithful representation $\pi_1 B \to \operatorname{PSL}(2,\real)$
is a cocompact surface Fuchsian group and hence all nontrivial elements
of $\Gamma$ are hyperbolic, that is, $|\operatorname{tr} X|>2$ for
$X\in \Gamma - \{ 1 \}$. Let $X$ be such an element.
Any hyperbolic element of $\operatorname{PSL}(2,\real)$ has precisely
two fixed points in $\overline{H}$, both of which lie in $\real \cup 
\{ \infty \}.$ In particular, $X$ has a finite, real fixed point $x$. Let
$v\in T_x (\real \cup \{ \infty \})$ be any nonzero tangent vector at this
fixed point. If
\[ X_\ast: T_x (\real \cup \{ \infty \}) \longrightarrow T_{X(x)} (\real \cup \{ \infty \}) =
   T_x (\real \cup \{ \infty \}) \]
denotes the differential of $X: \real \cup \{ \infty \} \to \real \cup \{ \infty \}$
at $x$, then $X_\ast (v)=\lambda v$ for some nonzero scalar $\lambda \in \real$
and
\[ |\lambda| \cdot \| v \|_x = \| \lambda v \|_x = \| X_\ast (v) \|_x = \| v \|_x, \]
where $\| \cdot \|_x$ is the putative $\Gamma$-invariant norm, evaluated at $x$.
Consequently, $X_\ast (v) = \pm v$.
Let $t$ be the standard coordinate in $\real \subset \real \cup \{ \infty \}$ and write
\[ X(t) = \frac{at+b}{ct+d},~ a,b,c,d\in \real,~ ad-bc=1,~ |a+d|>2. \]
Taking $v$ to be the standard basis vector $v=\partial_t$, we have
$X_\ast (\partial_t) = \frac{dX}{dt}(x)\cdot \partial_t$ and thus
\[ \left| \frac{dX}{dt}(x) \right|=1.  \]
In fact, however, the derivative of a hyperbolic M\"obius transformation in 
$\operatorname{PSL}(2,\real)$ at a finite, real fixed point is never $\pm 1$.
Indeed, if $c\not= 0,$ then
\[ x = \frac{1}{2c} (a-d \pm \sqrt{\Delta}), \]
where $\Delta$ is the discriminant $\Delta = \operatorname{tr}^2 X -4>0$, and
if $c=0,$ then $x = -b/(a-d).$ (Note that $c=0$ and $X$ hyperbolic implies that
$a\not= d$.) The derivative of $X$ is given by $dX/dt = (ct+d)^{-2}$.
Brief calculations verify that $(cx+d)^2$ cannot be $1$ at the above points $x$,
for hyperbolic $X$.
\end{proof}
This example illustrates that when modifying the structure group of a fiber bundle,
there is a tension between flatness and compactness of the structure group:
For a flat bundle with noncompact structure group, one can often reduce to a compact
group, but in doing so may be forced to give up flatness. Conversely, given a
compactly structured bundle which is not flat, one may sometimes gain
flatness at the expense of enlarging the structure group to a noncompact group.
For example, identifying the flat $\operatorname{PSL}(2,\real)$-bundle $\gamma$
with $S(TB)$, one may give $\gamma$ the structure group $\operatorname{SO}(2)$,
but one loses flatness ($S(TB)$ is not a flat $\operatorname{SO}(2)$-bundle).

\section{Equivariant Cohomology}
\label{sec.equivcohom}

We turn our attention to isometric actions of discrete, torsionfree groups. The actions
considered here
are usually not proper, and our $G$-spaces are generally not $G$-CW complexes.
Concerning condition (1) of Theorem \ref{thm.equivcohom} below, we remark again
that a nonproper action factors in many geometric situations through a proper action;
for instance, a discrete, torsionfree group may act nonproperly on a closed manifold,
but in such a way that the manifold can be endowed with an invariant Riemannian metric
--- in that case, the action factors through the (proper) action of the compact isometry group.
\begin{thm} \label{thm.equivcohom}
Let $F$ be an oriented, closed, smooth manifold and $G$ a discrete group,
whose Eilenberg-MacLane space $K(G,1)$ may be taken to be a closed,
smooth manifold. If for a smooth action of $G$ on $F$, \\

(1) the action factors through the proper, smooth action of a Lie group, \\

\noindent or \\

(2) $F$ is Riemannian and $G$ acts isometrically on $F$, \\

\noindent then the real $G$-equivariant cohomology of $F$ decomposes as
\[ H^k_G (F;\real) \cong \bigoplus_{p+q=k} H^p (G; \der^q (F;\real)), \]
where the $\der^q (F;\real)$ are $G$-modules determined by the action.
\end{thm}
\begin{proof}
If (1) is satisfied, that is, the action of $G$ on $F$ factors as
$G\to H\to \operatorname{Diffeo}(F)$ with $H$ a Lie group acting properly,
then $F$ can be equipped with an $H$-invariant Riemannian metric, by \cite{palaisexistslices}.
Then $H$, and thus also $G$, acts isometrically on $F$ so that it suffices
to prove the theorem assuming hypothesis (2). 
Let $B\simeq K(G,1)$ be an aspherical, closed, smooth manifold with
fundamental group $G$. The universal cover $\widetilde{B} \to B$
can serve as a model for the universal principal $G$-bundle
$EG\to BG$ because it is a principal $G$-bundle with contractible
total space $\widetilde{B}$, as follows from the fact that $\widetilde{B}$
is a simply connected CW-complex and $\pi_i (\widetilde{B})\cong
\pi_i (B)=0$ for $i\geq 2$. We recall that the Borel construction 
$F \times_G EG$ is the orbit space $(F\times EG)/G,$ where $G$
acts diagonally on the product $F\times EG$. The $G$-equivariant cohomology
of $F$ is the cohomology of the Borel construction,
\[ H^\bullet_G (F) = H^\bullet (F\times_G EG). \]
Using the model $\widetilde{B}\to B,$ this may be computed as
\[ H^\bullet_G (F) = H^\bullet (F\times_G \widetilde{B}). \]
The space $F\times_G \widetilde{B}$ is the total space of a flat fiber bundle
\[ \xymatrix@C=15pt@R=15pt{
F \ar@{^{(}->}[r] & F\times_G \widetilde{B} \ar[d] \\
& B,
} \]
whose projection is induced by the second-factor projection
$F\times \widetilde{B} \to \widetilde{B}$.
Since $G$ acts isometrically on $F$, Theorem \ref{thm.main1}
applies and we conclude that the Leray-Cartan-Lyndon spectral sequence
for real coefficients of the $G$-space $F$ collapses at the $E_2$-term.
In particular,
\begin{eqnarray*}
H^k_G (F;\real) & \cong & H^k (F\times_G \widetilde{B};\real) \cong
 \bigoplus_{p+q=k} H^p (B;\der^q (F;\real)) \\
& \cong & \bigoplus_{p+q=k} H^p (G;\der^q (F;\real)).
\end{eqnarray*}
\end{proof}
\begin{remark}
If $F \stackrel{i}{\longrightarrow} E \longrightarrow B$ is any fibration with
$B$ path-connected such that the restriction $i^\ast: H^\bullet (E;\real) \to H^\bullet (F;\real)$
is surjective (i.e. the fiber is ``totally nonhomologous to zero in $E$''), then the
action of $\pi_1 B$ on $H^\bullet (F;\real)$ is trivial
(see \cite[p. 51, Ex. 2]{hatcherss}) and the spectral sequence of the fibration
collapses at $E_2$, yielding the Leray-Hirsch theorem. Our results apply to situations
where $i^\ast$ needs not be surjective. In Section \ref{sec.flag}, we consider
an example of a certain $\intg^3$-action on the flag manifold $F_8 = U(8)/T^8$.
We compute (in two different ways) that $H^2_{\intg^3} (F_8)$ has rank $5$,
while $H^2 (F_8)$ has rank $7$. Thus $i^\ast$ is not surjective in this example.
Moreover, in the context of the present paper, the action of $\pi_1 B$ on
$H^\bullet (F;\real)$ is typically nontrivial.
\end{remark}

Let us discuss some immediate consequences of Theorem \ref{thm.equivcohom}. For a $G$-module
$H$, let $H^G = \{ v\in H ~|~ gv=v \text{ for all } g\in G \}$ denote the subspace
of invariant elements.
\begin{cor} \label{cor.1}
Let $F$ be an oriented, closed, connected, smooth manifold and $G$ a discrete
group as in Theorem \ref{thm.equivcohom}, acting smoothly on $F$ so that
hypothesis (1) or (2) is satisfied. Then there is a monomorphism
\[ H^k (G;\real)\oplus H^k (F;\real)^G \hookrightarrow H^k_G (F;\real) \]
for $k\geq 1$.
\end{cor}
\begin{proof}
For $k\geq 1$, the direct sum of Theorem \ref{thm.equivcohom}
contains the term $H^k (G;\der^0 (F;\real))$ and the term 
$H^0 (G;\der^k (F;\real))$. The former is isomorphic to $H^k (G;\real)$,
since $G$ acts trivially on $H^0 (F;\real)$ and $H^0 (F;\real)\cong \real$
as $F$ is connected. The latter is isomorphic to
$H^k (F;\real)^G;$ see the appendix.
\end{proof}
In particular, we obtain lower bounds for the ranks of the equivariant groups
in terms of the ranks of the group cohomology.
\begin{cor} \label{cor.2}
In the situation of Corollary \ref{cor.1}, the inequalities 
\[ \rk H^k (G) \leq \rk H^k_G (F) \]
hold for $k\geq 0$.
\end{cor}
If a $G$-space $F$ has a fixed point, then $\pi: F\times_G EG \to BG$ has a
section given by $[e] \mapsto [(f,e)],$ where $f$ is a fixed point.
Consequently, $\pi^\ast: H^\bullet (G)\to H^\bullet_G (F)$ is a 
monomorphism. Note that our results concern group actions that may be
fixed-point-free, even free. (See Example \ref{exple.zoncircle} below.)
Let $\cd_\real G$ denote the $\real$-cohomological dimension of a group $G$,
that is, the smallest $n\in \mathbb{N}\cup \{ \infty \}$ such that
$H^k (G;\real)$ vanishes for all $k>n$. For a topological space $X$,
let $\cd_\real X$ be the smallest $n\in \mathbb{N}\cup \{ \infty \}$ such that
the singular cohomology $H^k (X;\real)$ vanishes for all $k>n$.
\begin{cor} \label{cor.3}
In the situation of Corollary \ref{cor.1}, the inequality
\[ \cd_\real G \leq \cd_\real (F\times_G EG) \] holds.
\end{cor}
\begin{proof}
Suppose that $n=\cd_\real G$ is finite. Then $H^n (G;\real)$ is not zero.
By Corollary \ref{cor.2}, $H^n_G (F;\real)$ is not zero and it follows that
$\cd_\real (F\times_G EG)\geq n$. If $\cd_\real G=\infty,$ then for every
$n\in \mathbb{N}$, there exists an $N\geq n$ such that
$H^N (F\times_G EG;\real)\not= 0,$ whence $\cd_\real (F\times_G EG)=\infty$.
\end{proof}
Although the underlying spaces of the $G$-actions considered in this paper
are smooth manifolds, and hence can be given a (regular)
CW structure, no such structure can usually be found that would make the
$G$-space into a $G$-complex. A $G$-complex is a CW complex together
with a $G$-action which permutes the cells. If $X$ is a $G$-complex on
which $G$ acts freely, then $H^\bullet_G (X)\cong H^\bullet (X/G)$.
Thus, if $F$ were a free $G$-complex satisfying the hypotheses of
Corollary \ref{cor.1}, then, by Corollary \ref{cor.3},
\begin{equation} \label{equ.cdglcdfg}
\cd_\real G \leq \cd_\real (F/G).
\end{equation}
The following example of a free $G$-action shows that this inequality is
generally false in the context of Theorem \ref{thm.equivcohom}.
\begin{example} \label{exple.zoncircle}
Let $G=\intg$ act freely on $F=S^1$ by powers of a rotation by an angle which
is an irrational multiple of $2\pi$. The quotient topology on the orbit space
$S^1/\intg$ is the coarse topology, that is, the only open sets in $S^1 /\intg$
are the empty set and $S^1 /\intg$. The coarse topology on a set $X$ has
the property that any map $Y\to X$ is continuous. In particular, the map
$H: X\times I \to X$ given by $H(x,t)=x$ for $t\in [0,1)$ and 
$H(x,1)=x_0$ for all $x\in X$, where $x_0\in X$ is a base-point, is
continuous. Thus $X$ is homotopy equivalent to a point and therefore
acyclic. This shows that $\cd_\real (S^1 /\intg)=0.$ Since
$\cd_\real \intg = \cd_\real S^1 = 1,$ inequality (\ref{equ.cdglcdfg}) is
violated.
\end{example}
This also emphasizes that it is prudent to observe carefully the hypotheses
of the Vietoris-Begle mapping theorem in attempting to apply it to the
map $F\times_G EG \to F/G$ for a free action. 

Example \ref{exple.zoncircle} illustrates once again that the actions considered in the
present paper are generally not proper, since their orbits need not be closed.
Furthermore, the isotropy groups for a proper $G$-CW complex are compact
(so finite if $G$ is discrete). The isotropy groups arising for our actions can
be infinite. (Consider the trivial $\intg$-action, or a $\intg$-action that factors
through a finite cyclic group.) For a proper $G$-CW complex $X$, a result
of W. L\"uck \cite[Lemma 6.4]{lueckratcompkdiscgrps}, based on
\cite[Lemma 8.1]{lueckreichvarisco}, asserts that the projection
$X\times_G EG \to X/G$ induces an isomorphism
\[ H_n (X\times_G EG;\rat) \stackrel{\cong}{\longrightarrow}
  H_n (X/G;\rat). \]
For the actions arising in our Theorem \ref{thm.equivcohom}, the groups
$H_n (X\times_G EG;\rat)$ and $H_n (X/G;\rat)$ are generally not isomorphic,
as Example \ref{exple.zoncircle} shows.

Let us consider some specific groups.
\begin{cor} \label{cor.4}
If $\intg^n$ acts isometrically on an oriented, closed, connected, Riemannian manifold
$F$, then
\[ \rk H^k_{\intg^n} (F)\geq \binom{n}{k}, \]
with equality for $k=0$.
\end{cor}
\begin{proof}
The inequality follows from Corollary \ref{cor.2} by observing that we may take
$K(\intg^n, 1)=T^n,$ the $n$-torus, and $\rk H^k (T^n) =\binom{n}{k}.$
\end{proof}
\begin{cor} \label{cor.5}
If a discrete, integral Heisenberg group $\mathfrak{H}_n,$ $n$ a positive integer, acts
isometrically on an oriented, closed, connected, Riemannian manifold $F$, then
$\rk H^0_{\mathfrak{H}_n} (F)=1,$
\[ \rk H^k_{\mathfrak{H}_n} (F) \geq  2,
 \text{ for } k= 1,2, \] 
and $H^3_{\mathfrak{H}_n} (F;\real)$ does not vanish.
\end{cor}
\begin{proof}
Let $\mathfrak{H} (\real) \subset GL_3 (\real)$ be the continuous Heisenberg group,
i.e. the subgroup of upper triangular matrices of the form
\[ \begin{pmatrix} 1 & x & z \\ 0 & 1 & y \\ 0 & 0 & 1 \end{pmatrix},~
  x,y,z\in \real. \]
This is a contractible Lie group.
The discrete Heisenberg group $\mathfrak{H}_n$ can be described as the subgroup
of $\mathfrak{H} (\real)$ generated by the matrices
\[ X= \begin{pmatrix} 1 & 1 & 0 \\ 0 & 1 & 0 \\ 0 & 0 & 1 \end{pmatrix},~
   Y_n = \begin{pmatrix} 1 & 0 & 0 \\ 0 & 1 & n \\ 0 & 0 & 1 \end{pmatrix},~
   Z= \begin{pmatrix} 1 & 0 & 1 \\ 0 & 1 & 0 \\ 0 & 0 & 1 \end{pmatrix}. \]
It is a torsionfree, nilpotent group and a central extension of $\intg^2$ by $\intg$
with relations $[X,Y_n]=Z^n,$ $[X,Z]=1,$ $[Y_n, Z]=1$. Being a subgroup
of $\mathfrak{H}(\real),$ $\mathfrak{H}_n$ acts freely (and properly 
discontinuously and cocompactly) on $\mathfrak{H}(\real)$. Thus the quotient map
$\mathfrak{H}(\real) \to B$ is the universal cover of the orbit space
$B= \mathfrak{H}(\real) / \mathfrak{H}_n$ and $B$ is a closed, orientable,
smooth $3$-manifold, in fact, an orientable circle-bundle over the $2$-torus.
Moreover, $\pi_1 (B)=\mathfrak{H}_n$ and
$\pi_k (B) = \pi_k \mathfrak{H}(\real)=0$ for $k\geq 2$. Hence
$B = B\mathfrak{H}_n = K(\mathfrak{H}_n, 1)$ and we have
\begin{eqnarray*}
H_1 (\mathfrak{H}_n) & = & H_1 (B) = \pi_1 B / [\pi_1 B, \pi_1 B ] = 
    \mathfrak{H}_n / [\mathfrak{H}_n, \mathfrak{H}_n ] \\
& = & \langle X,Y_n ,Z ~|~ [X,Y_n]=1,~ [X,Z]=1,~ [Y_n,Z]=1,~ Z^n = 1 \rangle \\
& = & \intg^2 \oplus \intg/_n.
\end{eqnarray*}
(See also \cite[Chapter I, Section 3]{ademmilgram}.)
Thus
\[ \rk H^1 (\mathfrak{H}_n) = \rk H_1 (\mathfrak{H}_n) = 2. \]
By Poincar\'e duality, $\rk H^2 (\mathfrak{H}_n)=\rk H^1 (\mathfrak{H}_n).$ Furthermore,
$\rk H^3 (\mathfrak{H}_n)=1,$ as $B$ is connected, closed, and orientable.
The result follows from Corollary \ref{cor.2}.
\end{proof}

\section{An Example: Group Actions on Flag Manifolds}
\label{sec.flag}

We will compute the equivariant second real cohomology of a certain
isometric $\intg^3$-action on a (complete) flag manifold, using our
Theorem \ref{thm.equivcohom}. The result will then be independently confirmed by
employing the recursive scheme introduced in the appendix.
The flag manifold $F_n$ is the space of all decompositions
$\cplx v_1 \oplus \cdots \oplus \cplx v_n$ of $\cplx^n$ into lines,
where $\{ v_1, \ldots, v_n \}$ is an orthonormal basis of $\cplx^n$.
Let $U=U(n)$ be the unitary group and $T=T^n \subset U$ the maximal
torus given by the diagonal matrices. Writing each $v_i$ as a column
vector and using these as the columns of a matrix, we obtain an element
$u$ of $U$. If $\cplx v'_1 \oplus \cdots \oplus \cplx v'_n =
\cplx v_1 \oplus \cdots \oplus \cplx v_n,$ then $u' =ut$ with $t\in T$.
Thus $\{ v_1, \ldots, v_n \} \mapsto u$ induces a well-defined map
$F_n \to U/T,$ where $U/T$ denotes the space of left cosets $uT$.
This map is a diffeomorphism and exhibits $F_n$ as a homogeneous
space. The symmetric group $S_n$ acts on $F_n$ by permuting the lines
$\cplx v_i$. In terms of unitary matrices, this can be described as
follows. Since every permutation is a product of transpositions, it
suffices to describe the action of a transposition $\tau$. Let
$p_\tau$ be the corresponding permutation matrix. The normalizer
$N(T)$ of $T$ in $U$ consists of all generalized permutation matrices,
that is, matrices whose pattern of zero entries is the same as the
zero-pattern of a permutation matrix, but whose nonzero entries
can be any complex number of modulus one. In particular,
$p_\tau \in N(T)$. Using this, the map
$\overline{\tau}: U/T \to U/T$ induced by \emph{right}-multiplication
with $p_\tau,$ $\overline{\tau} (uT)= up_\tau T,$ is seen to be
well-defined. The diagram
\[ \xymatrix{
F_n \ar[r]^{\tau} \ar[d]_{\cong} &
 F_n \ar[d]^{\cong} \\
U/T \ar[r]^{\overline{\tau}} & U/T
} \]
commutes. Up to homotopy, $\overline{\tau}$ can also be described by
conjugation: Given an element $\nu \in N(T),$ conjugation by $\nu$
yields a well-defined map $c_\nu: U/T \to U/T,$
$c_\nu (uT)= \nu u \nu^{-1} T.$ While right multiplication by an
element generally only induces a map on $U/T$ if the element lies
in $N(T)$, left multiplication by \emph{any} element of $U$ induces
an automorphism of the homogeneous space. However, all these maps
$f_g: U/T \to U/T,$ $f_g (uT) = guT,$ are homotopic to the identity, since
in $U$ we can choose a path $g_t,$ $t\in [0,1],$ from $g$ to the
identity ($U$ is connected) and $uT \mapsto g_t uT$ is a homotopy
from $f_g$ to the identity. This shows that for $\nu = p_\tau,$
$c_\nu$ is homotopic to $\overline{\tau}$, as
$c_\nu (uT)= f_{p_\tau} (up_\tau T).$ (A similar argument cannot be
used to show that right multiplication $uT \mapsto u\nu T$ for fixed
$\nu \in N(T),$ is homotopic to the identity because a path $\nu_t$
from $\nu$ to the identity would have to be chosen within $N(T)$,
but $N(T)$ is disconnected.) Permutation matrices define a splitting
of the Weyl group $N(T)/T \cong S_n$ into $N(T)$. Then the
Weyl group acts on $U/T$ by conjugation, and this is up to homotopy
the action of $S_n$ on $F_n$ that permutes the lines. \\

The effect of this action on the real cohomology $H^\bullet (F_n)$ is
easily determined. Let $x\in H^2 (\cplx P^{n-1})$ be the standard
generator. There are $n$ projection maps $\pi_i: F_n \to \cplx P^{n-1}$
given by taking the $i$-th line of an orthogonal decomposition. With
$x_i = \pi^\ast_i (x)\in H^2 (F_n),$ we have
\[ H^\bullet (F_n) = \real [x_1, \ldots, x_n] / (e_i (x_1, \ldots, x_n)=0,~
  i=1,\ldots, n), \]
where $e_i (x_1, \ldots, x_n)$ is the $i$-th elementary symmetric
function in $x_1, \ldots, x_n$. The action of the Weyl group $S_n$
on $F_n$ induces the action on $H^\bullet (F_n)$ which permutes
the $x_i$.  \\

If a compact Lie group $G$ acts smoothly on a smooth manifold, then
that manifold can be equipped with a smooth $G$-invariant Riemannian
metric. Thus, as $S_n$ is finite, there exists a Riemannian metric on 
$F_n$ such that $S_n$ acts by isometries. We endow $F_n$ with such
a metric. \\

Next, let us determine which elements of $S_n$ preserve the orientation
of $F_n$. Let $\xi_i \in \Omega^2 (F_n)$ be a closed form representing
$x_i,$ $i=1,\ldots, n$. The flag manifold $F_n$ is a complex manifold of
complex dimension $d=n(n-1)/2.$ 
By \cite[Theorem 2.3]{hoffmanhomer},
there is a nonzero constant $C$ such that
for scalars $\lambda_1, \ldots, \lambda_n,$
\begin{equation} \label{equ.signperms}
\int_{F_n} (\lambda_1 \xi_1 + \ldots + \lambda_n \xi_n)^d =
 C \cdot \prod_{i<j} (\lambda_i - \lambda_j). 
\end{equation}
Thus if the $\lambda_i$ are all distinct, then 
$(\lambda_1 \xi_1 + \ldots + \lambda_n \xi_n)^d$ is a generator
of the real top-dimensional cohomology group $H^{2d} (F_n)\cong \real.$
When a permutation acts on the scalars $\lambda_i$ by permuting them,
then the right hand side of (\ref{equ.signperms}), and therefore also the
left hand side, changes sign if and only if the permutation is odd.
Consequently, the subgroup of orientation preserving automorphisms in
$S_n$ is the alternating group $A_n$. \\

To illustrate the use of our results, we turn to a specific example of a
$\intg^3$-action on $F_8 = U(8)/T^8,$ a manifold of real dimension $56$.
Let $R_1, R_2, R_3$ be the following elements of $A_8$:
\[ R_1 = (1,2)(3,4),~ R_2 = (1,3)(2,4),~ R_3 = (5,8,6). \]
Since these form a commuting set, they define an isometric,
orientation preserving $\intg^3$-action on $F_8$. The set
$\{ R_1, R_2, R_3 \}$ generates an abelian subgroup of order $12$ in 
$A_8$. We shall determine the equivariant cohomology group
$H^2_{\intg^3} (F_8)$. The classifying space is $B\intg^3 = T^3.$
The group $H^2 (F_8)$ has rank $7$ and is linearly generated
by $x_1, \ldots, x_8$ subject to the relation
$e_1 = x_1 + \ldots + x_8=0$. An element 
\[ y = \lambda_1 x_1 + \ldots + \lambda_8 x_8 \in
  H^0 (T^3; \der^2 (F_8)) \cong H^2 (F_8)^{\intg^3} \]
must satisfy the system of equations
\[ \begin{array}{lcl}
(\lambda_1 - \lambda_2)x_1 + (\lambda_2 - \lambda_1)x_2 +
(\lambda_3 - \lambda_4)x_3 + (\lambda_4 - \lambda_3)x_4 & = &
 \lambda (x_1 + \ldots + x_8), \\
(\lambda_1 - \lambda_3)x_1 + (\lambda_2 - \lambda_4)x_2 +
(\lambda_3 - \lambda_1)x_3 + (\lambda_4 - \lambda_2)x_4 & = &
 \lambda' (x_1 + \ldots + x_8), \\
(\lambda_5 - \lambda_6)x_5 + (\lambda_6 - \lambda_8)x_6 +
(\lambda_8 - \lambda_5)x_8 & = &
 \lambda'' (x_1 + \ldots + x_8), 
\end{array} \] 
for some $\lambda, \lambda', \lambda'' \in \real,$ if we regard 
$x_1, \ldots, x_8$ as linearly independent generators of an
$8$-dimensional vector space $V$ with $V/\real e_1=H^2 (F_8)$.
Thus $y$ is of the form
\begin{eqnarray*}
y & = & \lambda_1 (x_1 + \ldots + x_4) + \lambda_5 (x_5 + x_6 + x_8)+
 \lambda_7 x_7 \\
& = & (\lambda_1 - \lambda_7)(x_1 + \ldots + x_4) +
 (\lambda_5 - \lambda_7)(x_5 + x_6 + x_8).
\end{eqnarray*}
The twisted group $H^0 (T^3; \der^2 (F_8))$ thus has rank $2$ and basis
$\{ x_1 + x_2 + x_3 + x_4, x_5 + x_6 + x_8 \}.$ Since $H^1 (F_8)=0,$
we have $H^1 (T^3; \der^1 (F_8))=0.$ Furthermore, as $\der^0 (F_8)$
is a constant local system of rank one on $T^3,$ the group
$H^2 (T^3; \der^0 (F_8)) = H^2 (T^3;\real)$ has rank $3$ generated by the dual
basis of the homology basis
$S^1 \times S^1 \times \pt,$ $S^1 \times \pt \times S^1,$ and
$\pt \times S^1 \times S^1.$ By Theorem \ref{thm.equivcohom},
\begin{equation} \label{equ.resultexple}
\begin{array}{rcl}
H^2_{\intg^3} (F_8) & = & H^0 (T^3; \der^2 (F_8)) \oplus
  H^1 (T^3; \der^1 (F_8)) \oplus H^2 (T^3; \der^0 (F_8)) \\
& = & H^2 (F_8)^{\intg^3} \oplus 0 \oplus H^2 (T^3) \\
& \cong & \real^5.
\end{array} 
\end{equation}

We shall confirm this result by applying the recursive scheme of the appendix.
Let $B_1 = \real^1 \times_\intg F_8$ be the Borel construction of the
$\intg$-action on $F_8$ by powers of $R_1,$ 
$B_2 = \real^1 \times_\intg B_1$ the Borel construction of the
$\intg$-action on $B_1$ by powers of $\oR_2,$ and
$B_3 = \real^1 \times_\intg B_2$ the Borel construction of the
$\intg$-action on $B_2$ by powers of $\oR_3$ (notation as in the
appendix). As explained in the appendix, $H^2_{\intg^3} (F_8) =
H^2 (B_3)$ can then be found by carrying out the following steps: \\

\noindent \begin{tabular}{lp{12cm}}
1. & Determine the action of $R^\ast_1, R^\ast_2$ and $R^\ast_3$
 on both $H^0 (F_8)$ (trivial) and $H^2 (F_8).$
 (Note $H^1 (F_8)=0$.) \tabularnewline
2. & Compute $(\Der_\intg/\Princ_\intg)H^0 (F_8)$ and $H^2 (F_8)^\intg$.
  \tabularnewline
3. & Compute the action of $\oR^\ast_2$ and $\oR^\ast_3$ on both
 $(\Der_\intg/\Princ_\intg) H^0 (F_8)$ and $H^2 (F_8)^\intg$. \tabularnewline
4. & Step $3$ determines the action of $\oR^\ast_2$ and $\oR^\ast_3$ on both
 $H^1 (B_1)$ and $H^2 (B_1)$. The action of $\oR^\ast_2,$ $\oR^\ast_3$
 on $H^0 (B_1)$ is trivial. \tabularnewline
5. & Compute $(\Der_\intg/\Princ_\intg) H^0 (B_1),$
  $(\Der_\intg/\Princ_\intg) H^1 (B_1),$ $H^1 (B_1)^\intg$ and 
 $H^2 (B_1)^\intg$. \tabularnewline
6. & Using the results of steps $4$ and $5$, determine the action of
 $\oR^\ast_3$ on the four groups computed in step $5$. \tabularnewline
7. & The result of step $6$ determines the action of $\oR^\ast_3$ on
 $H^1 (B_2)$ and $H^2 (B_2)$. \tabularnewline
8. & Compute $(\Der_\intg/\Princ_\intg) H^1 (B_2)$ and
 $H^2 (B_2)^\intg$. \tabularnewline
9. & The direct sum of the two groups obtained in step $8$ is $H^2 (B_3)$.
\end{tabular} \\

\noindent We shall now carry out these steps. 

\emph{Step 1.} For $y=\lambda_1 x_1 +\ldots +\lambda_8 x_8 \in H^2 (F_8),$
 we have 
\begin{eqnarray*}
R^\ast_1 (y) & = & \lambda_2 x_1 + \lambda_1 x_2 + \lambda_4 x_3 +
 \lambda_3 x_4 + \lambda_5 x_5 + \ldots + \lambda_8 x_8, \\
R^\ast_2 (y) & = & \lambda_3 x_1 + \lambda_4 x_2 + \lambda_1 x_3 +
 \lambda_2 x_4 + \lambda_5 x_5 + \ldots + \lambda_8 x_8, \\
R^\ast_3 (y) & = & \lambda_1 x_1 + \ldots +\lambda_4 x_4 + \lambda_6 x_5 +
 \lambda_8 x_6 + \lambda_7 x_7 + \lambda_5 x_8. 
\end{eqnarray*}

\emph{Step 2.} If a group $G$ acts trivially on a vector space $V$, then
 $\Princ_G (V)=0$. Moreover, if $G=\intg$ acts trivially, then
 $\Der_\intg (V)\cong V,$ where the isomorphism is given by
 $f \mapsto f(1)$. We conclude that
 $(\Der_\intg/\Princ_\intg)H^0 (F_8) \cong H^0 (F_8).$ 
 If $y\in H^2 (F_8)$ and $R^\ast_1 (y)=y,$ then by step 1,
 $\lambda_1 = \lambda_2,$ $\lambda_3 = \lambda_4$. The set
 $\{ x_1 + x_2, x_3 + x_4, x_5, x_6, x_7 \}$ is a basis for
 $H^2 (F_8)^\intg$ so that $H^2 (F_8)^\intg \cong \real^5$.

\emph{Step 3.} The automorphisms $\oR^\ast_2,$ $\oR^\ast_3$
 act trivially on $(\Der_\intg/\Princ_\intg)H^0 (F_8)$. The action of
 $\oR^\ast_2$ on $H^2 (F_8)^\intg$ is given by
 \[ x_1 + x_2 \mapsto x_3 + x_4,~
    x_3 + x_4 \mapsto x_1 + x_2,~ x_5 \mapsto x_5,~
  x_6 \mapsto x_6,~ x_7 \mapsto x_7, \]
while $\oR^\ast_3$ acts on $H^2 (F_8)^\intg$ by
 \[ x_1 + x_2 \mapsto x_1 + x_2,~
    x_3 + x_4 \mapsto x_3 + x_4,~ x_5 \mapsto x_8 =
 -(x_1 + x_2) - (x_3 + x_4) -x_5 -x_6 -x_7, \]
\[ x_6 \mapsto x_5, x_7 \mapsto x_7. \]

\emph{Step 4.} Since by Proposition \ref{prop.hqe},
\[ H^1 (B_1) = H^1 (F_8)^\intg \oplus
 \frac{\Der_\intg}{\Princ_\intg} H^0 (F_8) =
 \frac{\Der_\intg}{\Princ_\intg} H^0 (F_8), \]
 $\oR^\ast_2$ and $\oR^\ast_3$ both act trivially on $H^1 (B_1)$,
 using step $3$. Again by Proposition \ref{prop.hqe},
\[ H^2 (B_1) = H^2 (F_8)^\intg \oplus
 \frac{\Der_\intg}{\Princ_\intg} H^1 (F_8) = H^2 (F_8)^\intg. \]
The action of $\oR^\ast_2, \oR^\ast_3$ on $H^2 (B_1)$ is thus given
by the assignments listed in the previous step. 

\emph{Step 5.} We have $(\Der_\intg/\Princ_\intg)H^0 (B_1) \cong
 H^0 (B_1),$ since $\oR^\ast_2$ acts trivially on $H^0 (B_1).$
Since $\oR^\ast_2$ acts trivially on $H^1 (B_1),$ $(\Der_\intg/\Princ_\intg)H^1 (B_1) \cong
 H^1 (B_1)$ and $H^1 (B_1)^\intg = H^1 (B_1).$ If
\[ y= \lambda_1 (x_1 + x_2) + \lambda_3 (x_3 + x_4) +
 \lambda_5 x_5 + \lambda_6 x_6 + \lambda_7 x_7 \]
is any element of $H^2 (B_1)$ and $\oR^\ast_2 (y)=y,$ then
$\lambda_1 = \lambda_3$. A basis of $H^2 (B_1)^\intg$ is thus
given by $\{ x_1 + x_2 + x_3 + x_4, x_5, x_6, x_7 \}$;
$H^2 (B_1)^\intg \cong \real^4.$

\emph{Step 6.} The map $\oR^\ast_3$ acts trivially on
$(\Der_\intg/\Princ_\intg)H^0 (B_1),$ $(\Der_\intg/\Princ_\intg)H^1 (B_1)$
and $H^1 (B_1)^\intg$. The action of $\oR^\ast_3$ on 
$H^2 (B_1)^\intg$ is given by
\[ \begin{array}{rcl}
x_1 + \ldots + x_4 & \mapsto & x_1 + \ldots + x_4, \\
x_5 & \mapsto & -(x_1 + \ldots +x_4) -x_5 - x_6 - x_7 \\
x_6 & \mapsto & x_5 \\ x_7 & \mapsto & x_7.
\end{array} \]

\emph{Step 7.} As
\[ H^1 (B_2) = H^1 (B_1)^\intg \oplus
 \frac{\Der_\intg}{\Princ_\intg} H^0 (B_1), \]
$\oR^\ast_3$ acts trivially on $H^1 (B_2)$ by step $6$ and
Proposition \ref{prop.hqe}. Let $\{ X \}$ be a basis for
\[ \frac{\Der_\intg}{\Princ_\intg} H^1 (B_1) = H^1 (B_1) =
 \frac{\Der_\intg}{\Princ_\intg} H^0 (F_8) = H^0 (F_8). \]
Since
\[ H^2 (B_2) = H^2 (B_1)^\intg \oplus
 \frac{\Der_\intg}{\Princ_\intg} H^1 (B_1), \]
a basis for $H^2 (B_2)$ is given by
\[ \{ x_1 + \ldots + x_4, x_5, x_6, x_7, X \}, \]
and $\oR^\ast_3$ transforms this basis according to step $6$ and
$X\mapsto X$.

\emph{Step 8.} We have $(\Der_\intg/\Princ_\intg) H^1 (B_2) \cong
 H^1 (B_2),$ for $\oR^\ast_3$ acts trivially on $H^1 (B_2)$. If
\[ y = \lambda_1 (x_1 +\ldots +x_4) + \lambda_5 x_5 +
 \lambda_6 x_6 + \lambda_7 x_7 + \lambda X \]
is any element of $H^2 (B_2)$ such that $\oR^\ast_3 (y)=y,$ then
$\lambda_5 = \lambda_6 =0.$ A basis of $H^2 (B_2)^\intg$ is given
by $\{ x_1 +\ldots +x_4, x_7, X \}$; $H^2 (B_2)^\intg \cong \real^3$.

\emph{Step 9.} Proposition \ref{prop.hqe}, Lemma 
\ref{lem.identifyborelconstr}, and
\[ \frac{\Der_\intg}{\Princ_\intg} H^1 (B_2) \cong H^1 (B_2) \cong
  H^1 (B_1)^\intg \oplus \frac{\Der_\intg}{\Princ_\intg} H^0 (B_1) \hspace{2cm} \]
\[ \hspace{3cm} \cong
 H^1 (B_1) \oplus H^0 (B_1) \cong H^0 (F_8) \oplus \real \cong
 \real^2 \]
show that
\[ H^2_{\intg^3} (F_8) = H^2 (B_3) \cong
 H^2 (B_2)^\intg \oplus \frac{\Der_\intg}{\Princ_\intg} H^1 (B_2) 
 \cong \real^5, \]
in agreement with (\ref{equ.resultexple}).

\section{Appendix: The Case of the Free Abelian Group}
\label{appendix}

We shall describe a recursive scheme for calculating the equivariant cohomology
of an isometric action of $\intg^n$ on a closed, oriented, path-connected, Riemannian manifold
$F$. We begin with some remarks on calculating the action of the homomorphism
induced on low-degree local system cohomology by a fiber preserving map.
We shall think of a local system $\der$ on a base space $B$ as a functor
$\der: \Pi_1 (B) \to \real-\operatorname{MOD}$ from the fundamental groupoid
$\Pi_1 (B)$ to the category $\real-\operatorname{MOD}$ of real vector spaces
and linear maps. Thus $\der$ assigns to every $b\in B$ a real vector space
$\der (b)$ and to every homotopy class $[\omega ]\in \pi_1 (B; b_1, b_2),$
where $\omega$ is a path $\omega: I\to B,$ $\omega (0)=b_1,$
$\omega (1)=b_2,$ an isomorphism
$\der [\omega ]: \der (b_2) \to \der (b_1).$ The functor $\der$ satisfies
$\der([\omega ][\eta ]) = \der [\omega ] \circ \der [\eta].$ Let
$e_0, e_1,\ldots$ be the canonical orthonormal basis for $\real^\infty$.
The standard simplex $\Delta^p$ is the convex hull of $\{ e_0,\ldots,
e_p \}$. Let $C^p (B;\der)$ be the set of all functions $c$, which assign
to each singular simplex $u: \Delta^p \to B$ an element
$c(u) \in \der (u(e_0)).$ This set $C^p (B;\der)$ is an abelian group under
addition of function values. Let $\sigma_u$ be the homotopy class of the
path $t\mapsto u((1-t)e_1 + te_0),$ defining an isomorphism
$\der (\sigma_u): \der (u(e_0)) \stackrel{\cong}{\longrightarrow}
\der (u(e_1)).$ Taking cohomology with respect to the coboundary operator
$\delta: C^p (B;\der) \to C^{p+1} (B;\der)$ given by
\[ (-1)^p \delta c (u) = \der (\sigma_u)^{-1} c(\partial_0 u) +
 \sum_{i=1}^{p+1} (-1)^i c(\partial_i u) \]
yields $H^p (B;\der)$, the cohomology of $B$ with coefficients in $\der$.

The recursive scheme concerning $\intg^n$-actions
only requires being able to compute in the degrees $p=0$ and $p=1$.
Assume that $B$ is path-connected and endowed with a base-point
$b_0$. Let $G = \pi_1 (B,b_0)$ be the fundamental group and let
$\der (b_0)^G$ denote the $G$-invariants of $\der (b_0)$, that is,
\[ \der (b_0)^G = \{ v\in \der (b_0) ~|~ g \cdot v = v
 \text{ for all } g\in G \}. \]
Here we wrote $g\cdot v = \der (g)(v),$
$\der (g): \der (b_0)\to \der (b_0),$ $g\in \pi_1 (B;b_0, b_0) = G.$
Let us recall the well-known isomorphism
\[ \kappa^G: \der (b_0)^G \stackrel{\cong}{\longrightarrow}
 H^0 (B;\der) = Z^0 (B;\der). \]
For every $b\in B,$ choose a path-class $\xi_b \in 
\pi_1 (B; b, b_0)$ starting at $b$ and ending at the base-point $b_0$.
Define a map
\[ \kappa: \der (b_0) \longrightarrow C^0 (B;\der) \]
by
\[ \kappa (v)(b) = \der (\xi_b)(v) \in \der (b), \]
where $\der (\xi_b): \der (b_0) \stackrel{\cong}{\longrightarrow}
 \der (b).$ (Zero-simplices $u$ are points $b$ in $B$.)
If $v\in \der (b_0)^G \subset \der (b_0)$ is a $G$-invariant vector,
then $\delta \kappa (v)=0,$ that is, $\kappa (v)$ is a cocycle.
Hence $\kappa$ restricts to a map
\[ \kappa^G: \der (b_0)^G \longrightarrow Z^0 (B;\der) =
 H^0 (B;\der). \]
If $\kappa^G (v)=0,$ then $\der (\xi_b)(v)=0$ and thus $v=0$
as $\der (\xi_b)$ is an isomorphism. This shows that $\kappa^G$
is injective. If $c\in Z^0 (B;\der)$ is a cocycle, then
$v = c(b_0)$ is a $G$-invariant vector with $\kappa^G (v)=c.$
Thus $\kappa^G$ is surjective as well, hence an isomorphism.
This isomorphism will be used to compute $H^p (B;\der)$ for
$p=0$. 

The group $H^1 (B;\der)$ ($p=1$) will be computed as derivations
modulo principal derivations. A \emph{derivation} is a function
$f: G\to V$, where $V$ is a real $G$-vector space, such that
$f(gh) = f(g) + g\cdot f(h)$ for all $g,h \in G$. The set of all
functions from $G$ to $V$ is a real vector space under
pointwise addition and scalar multiplication. The derivations
form a linear subspace $\Der_G (V)$. A \emph{principal}
derivation is a derivation $f$ of the form $f(g)=g\cdot v - v$
for some $v\in V$ and all $g\in G$. The principal derivations
form a subspace $\Princ_G (V) \subset \Der_G (V).$ We shall 
recall the well-known isomorphism
\[ \theta: \frac{\Der_G \der (b_0)}{\Princ_G \der (b_0)}
 \stackrel{\cong}{\longrightarrow} H^1 (B;\der). \]
Observe that $H^1 (B;\der)$ can be computed by only
considering singular $1$-simplices that close up. Given a
derivation $f: G\to \der (b_0),$ set
\[ \theta [f] (u) = f (\sigma^{-1}_u) \]
on a closed $1$-simplex $u: (\Delta^1, \partial \Delta^1)\to
(B,b_0).$ If $f$ is principal, then $u\mapsto f(\sigma^{-1}_u)$
is a coboundary. Thus $\theta$ is well-defined and it is
readily verified that it is an isomorphism. \\

Let us discuss the naturality of the above constructions.
Suppose $\der$ and $\bG$ are local coefficient systems on $B$.
A homomorphism $\phi: \der \to \bG$ of local systems is a
natural transformation of functors. Thus for every point
$b\in B$ there is a linear map $\phi (b): \der (b)\to \bG (b)$
such that
\[ \xymatrix{
\der (b_2) \ar[r]^{\der (\omega)} \ar[d]_{\phi (b_2)} &
 \der (b_1) \ar[d]^{\phi (b_1)} \\
\bG (b_2) \ar[r]^{\bG (\omega)} & \bG (b_1)
} \]
commutes for every $\omega \in \pi_1 (B;b_1, b_2).$
This implies in particular that
$\phi (b_0): \der (b_0) \to \bG (b_0)$ is a $G$-equivariant map.
A homomorphism $\phi$ of local systems induces a cochain map
\[ \phi^\#: C^p (B;\der) \longrightarrow C^p (B;\bG) \]
by
\[ (\phi^\# c)(u) = \phi (u(e_0))(c(u)). \]
Here, $c \in C^p (B;\der)$ is a cochain, $u: \Delta^p \to B$ is a
singular simplex, and
$\phi (u(e_0)): \der (u(e_0))\to \bG (u(e_0)),$
$c(u) \in \der (u(e_0)).$ This cochain map in turn induces a map
\[ \phi^\ast: H^p (B;\der) \longrightarrow H^p (B;\bG) \]
on cohomology.

For $p=0,1$ we wish to understand $\phi^\ast$ in terms of the
above identifications $\kappa^G$ and $\theta$. Let us
discuss $p=0$. The image of a $G$-invariant vector
$v\in \der (b_0)$ under $\phi (b_0): \der (b_0)\to \bG (b_0)$
is again $G$-invariant, since $\phi (b_0)$ is $G$-equivariant.
Therefore, $\phi (b_0)$ restricts to a map
\[ \phi^\ast = \phi (b_0)|: \der (b_0)^G \longrightarrow
  \bG (b_0)^G. \]
The calculation
\begin{eqnarray*}
\kappa^G (\phi^\ast v)(b) 
& = & \kappa^G (\phi (b_0)(v))(b) \\
& = & \bG (\xi_b) (\phi (b_0)(v)) \\
& = & \phi (b) (\der (\xi_b)(v)) \\
& = & \phi (b) (\kappa^G (v)(b)) \\
& = & (\phi^\# \kappa^G (v))(b)
\end{eqnarray*}
shows that
\[ \xymatrix{
\der (b_0)^G \ar[r]^{\kappa^G} \ar[d]_{\phi^\ast} &
  H^0 (B;\der) \ar[d]^{\phi^\ast} \\
\bG (b_0)^G \ar[r]^{\kappa^G} & H^0 (B;\bG)
} \]
commutes. Thus $\kappa^G$ is a natural isomorphism.
Let us turn to $p=1$ and $\theta$. Composing a derivation
$f: G\to \der (b_0)$ with $\phi (b_0): \der (b_0)\to
 \bG (b_0)$ yields a derivation $\phi (b_0)f: G\to
\bG (b_0).$ Hence, composition with $\phi (b_0)$ defines
a map 
\[ \phi^\ast: \Der_G \der (b_0) \longrightarrow
   \Der_G \bG (b_0). \]
Moreover, if $f$ is principal, say $f(g)=gv-v =
\der (g)(v)-v$ for all $g\in G,$ then
\begin{eqnarray*}
(\phi (b_0)f)(g)
& = & \phi (b_0)(\der (g)(v)-v) \\
& = & \bG (g)(w) - w
\end{eqnarray*}
with $w= \phi (b_0)(v).$ Thus $\phi^\ast (f)$ is again principal
and $\phi^\ast$ induces a map
\[ \phi^\ast: \frac{\Der_G}{\Princ_G} \der (b_0) \longrightarrow
  \frac{\Der_G}{\Princ_G} \bG (b_0). \]
For $[f] \in \Der_G \der (b_0)/ \Princ_G \der (b_0),$ we have
\begin{eqnarray*}
(\phi^\ast \theta [f])(u)
& = & \phi (b_0)(\theta [f] (u)) \\
& = & \phi (b_0)(f(\sigma^{-1}_u)) \\
& = & \theta [\phi (b_0) \circ f](u) \\
& = & (\theta \phi^\ast [f])(u),
\end{eqnarray*}
which proves that
\[ \xymatrix{
\frac{\Der_G}{\Princ_G} \der (b_0) \ar[r]^{\theta} \ar[d]_{\phi^\ast} &
  H^1 (B;\der) \ar[d]^{\phi^\ast} \\
\frac{\Der_G}{\Princ_G} \bG (b_0) \ar[r]^{\theta} & H^1 (B;\bG)
} \]
commutes. \\

Let $(F, g_F)$ be a closed, oriented, path-connected, Riemannian manifold.
Let $R$ and $S$ be commuting orientation preserving isometries of $F$,
$RS=SR,$ defining a $\intg^2$-action on $F$. We endow the real line
$\real^1$ with the canonical metric $g_1 = dt^2$ and give
$\real^1 \times F$ the product metric $g_1 + g_F$. Using $S$, we define
a diffeomorphism $\overline{S}: \real^1 \times F \to \real^1 \times F$
by $\overline{S} (t,x) = (t+1, S(x)),$ $t\in \real^1,$ $x\in F.$
This map defines a $\intg$-action on $\real^1 \times F.$ Let $E$
be the orbit space of this action. Since the action is properly discontinuous,
the quotient map $\real^1 \times F \to E$ is a regular covering map and
$E$ is a smooth manifold, the mapping torus of $S$. If $M,N$ are
Riemannian manifolds and $f: M\to M,$ $g:N\to N$ isometries, then
$f\times g: M\times M \to N\times N$ is an isometry for the product metric.
Since the translation $t\mapsto t+1$ is an isometry and $S$ is an
isometry, we conclude that $\overline{S}$ is an isometry. Thus $\intg$
acts by isometries on $\real^1 \times F$ and there exists a unique
metric $g_E$ on $E$ such that the quotient map 
$\real^1 \times F \to E$ is a local isometry. The projection
$(t,x) \mapsto t$ induces a fiber bundle projection $p: E\to S^1$
such that
\[ \xymatrix{
\real^1 \times F \ar[r] \ar[d] & E \ar[d]^p \\
\real^1 \ar[r] & \real^1 / \intg = S^1 
} \]
commutes. Using the isometry $R$, define a fiber-preserving diffeomorphism
$\overline{R}: E\to E$ by $\overline{R} [t,x] = [t, R(x)].$ This is
well-defined as $R$ and $S$ commute:
\[ \overline{R} [t+1, S(x)] = [t+1, RS(x)] = [t+1, SR(x)]
 = [t, R(x)] = \overline{R} [t,x]. \]
Moreover, $\overline{R}$ is an isometry of $(E,g_E)$, since
$(t,x) \mapsto (t, R(x))$ is an isometry of $\real^1 \times F$.
Our next goal is to describe the induced map
$\overline{R}^\ast: H^\bullet (E)\to H^\bullet (E)$ in terms of a
calculation of $H^\bullet (E)$ through the Leray-Serre spectral
sequence of $p: E\to S^1$. Let $K\subset H^q (E)$ be the subspace
\[ K = \ker (H^q (E)\longrightarrow H^q (F)). \]
Since $S$ is orientation preserving, $E$ is orientable and receives
an orientation from the canonical orientation of $\real^1$ and the
given orientation of $F$. Thus $E$ is a closed, oriented manifold
and Hodge theory applies. On the vector space $\Omega^q (E)$
of smooth $q$-forms on $E$, an inner product is given by
\[ \langle \omega, \eta \rangle = \int_E \omega \wedge 
  \ast \eta. \]
The Hodge theorem asserts that with respect to this inner product,
there is an orthogonal decomposition
\[ \Omega^q (E) = \Delta \Omega^q (E) \oplus \Harm^q (E), \]
where $\Delta$ is the Hodge Laplacian on $E$ and $\Harm^q (E)$ are
the harmonic $q$-forms. The inner product
restricts to an inner product on $\Harm^q (E)$, which yields an
inner product on $H^q (E)$ via the isomorphism
$\Harm^q (E) \stackrel{\cong}{\longrightarrow} H^q (E)$
induced by the inclusion. Let $K^\perp$ be the orthogonal
complement of $K$ in $H^q (E)$ with respect to this inner product.

\begin{lemma} \label{lem.rstarinv}
The subspaces $K, K^\perp \subset H^q (E)$ are both
$\oR^\ast$-invariant.
\end{lemma}
\begin{proof}
Let $i:F \hookrightarrow E$ be the inclusion of the fiber over the base-point.
By the definition of $\oR$, the square
\[ \xymatrix{
F \ar[r]^R \ar@{^{(}->}[d]_i & F \ar@{^{(}->}[d]_i \\
E \ar[r]^{\oR} & E
} \]
commutes. On cohomology, it induces the commutative diagram
\[ \xymatrix{
H^q (F) & H^q (F) \ar[l]_{R^\ast} \\
H^q (E) \ar[u]^{i^\ast} & H^q (E). \ar[l]^{\oR^\ast} \ar[u]_{i^\ast}
} \] 
If $v\in K = \ker i^\ast$, then
\[ i^\ast \oR^\ast (v) = R^\ast i^\ast (v)=0. \]
Thus $\oR^\ast (v)\in K$ and $K$ is $\oR^\ast$-invariant.

In fact, $K$ is invariant under $(\oR^\ast)^{-1}$ as well: The inverse
of $\oR$ is given by $\oR^{-1} [t,x] = [t, R^{-1} (x)].$ Thus
\[ \xymatrix{
F \ar[r]^{R^{-1}} \ar@{^{(}->}[d]_i & F \ar@{^{(}->}[d]_i \\
E \ar[r]^{\oR^{-1}} & E
} \]
commutes, inducing
\[ \xymatrix{
H^q (F) & H^q (F) \ar[l]_{(R^\ast)^{-1}} \\
H^q (E) \ar[u]^{i^\ast} & H^q (E). \ar[l]^{(\oR^\ast)^{-1}} \ar[u]_{i^\ast}
} \] 
If $v\in K$, then
\[ i^\ast (\oR^\ast)^{-1} (v) = (R^\ast)^{-1} i^\ast (v)=0, \]
proving the  $(\oR^\ast)^{-1}$-invariance of $K$.

Since $\oR$ is an orientation preserving isometry, its induced map commutes with the
Hodge star, $\oR^\ast \circ \ast = \ast \circ \oR^\ast$. Thus, for harmonic forms
$\omega, \eta \in \Harm^q (E) \cong H^q (E),$
\begin{eqnarray*}
\langle \oR^\ast \omega, \oR^\ast \eta \rangle 
& = & \int_E \oR^\ast \omega \wedge \ast (\oR^\ast \eta) =
   \int_E \oR^\ast \omega \wedge \oR^\ast (\ast \eta) \\
& = & \int_E \oR^\ast (\omega \wedge \ast \eta) =
  \int_E \omega \wedge \ast \eta \\ 
& = & \langle \omega, \eta \rangle.
\end{eqnarray*}
Consequently, $\oR^\ast$ is an orthogonal transformation on $H^q (E)$, and so is
$(\oR^\ast)^{-1}$. For $v\in K$ and $w\in K^\perp,$
\[ \langle \oR^\ast w, v \rangle =
 \langle (\oR^\ast)^{-1} \oR^\ast (w), (\oR^\ast)^{-1} (v) \rangle =
  \langle w, (\oR^\ast)^{-1} (v) \rangle =0, \]
since $(\oR^\ast)^{-1} (v)\in K.$ Hence $\oR^\ast w \in K^\perp$ and
$K^\perp$ is $\oR^\ast$-invariant.
\end{proof}

By Lemma \ref{lem.rstarinv}, $\oR^\ast$ restricts to maps
\[ R^\ast_K: K\longrightarrow K,~ R^\ast_\perp: K^\perp \longrightarrow K^\perp. \]
It follows that $\oR^\ast: H^q (E)\to H^q (E)$ splits as an orthogonal sum
\begin{equation} \label{equ.rbarisrkplusrperp}
\oR^\ast = R^\ast_\perp  \oplus R^\ast_K.
\end{equation}
Orthogonal projection defines a map $H^q (E)\to K^\perp$. There is a unique
isomorphism $\tau: K^\perp \to H^q (E)/K$ such that the diagram
\[ \xymatrix{
0 \ar[r] & K \ar[r] \ar@{=}[d] & H^q (E) \ar[r] \ar@{=}[d] & K^\perp \ar[r] 
\ar@{..>}[d]^{\tau} &
  0 \\
0 \ar[r] & K \ar[r] & H^q (E) \ar[r] & H^q (E)/K \ar[r] & 0,
} \]
with exact rows, commutes. Since $K$ is $\oR^\ast$-invariant,
$\oR^\ast$ induces a map
$\oR^\ast_Q: H^q (E)/K \to H^q (E)/K.$

\begin{lemma} \label{lem.taurperprqcomm}
The diagram
\[ \xymatrix{
K^\perp \ar[r]^<<<<{\cong}_<<<<{\tau} \ar[d]_{R^\ast_\perp} &
H^q (E)/K \ar[d]^{\oR^\ast_Q} \\
K^\perp \ar[r]^<<<<{\cong}_<<<<{\tau} & H^q (E)/K
} \]
commutes.
\end{lemma}
\begin{proof}
The diagram
\[ \xymatrix{
K^\perp \ar[rr]^{\tau} \ar[ddd]_{R^\ast_\perp} & & 
  H^q (E)/K \ar[ddd]^{\oR^\ast_Q} \\
& H^q (E) \ar[lu]^{\operatorname{proj}}
 \ar[ru]_{\operatorname{quot}} \ar[d]^{\oR^\ast} & \\
& H^q (E) \ar[ld]_{\operatorname{proj}} 
 \ar[rd]^{\operatorname{quot}} & \\
K^\perp \ar[rr]^{\tau} & & H^q (E)/K
} \]
commutes. Given $v\in K^\perp \subset H^q (E),$
we have $\operatorname{proj}(v)=v$ and thus
$\tau (v) = \operatorname{quot}(v).$ Therefore,
\begin{eqnarray*}
\oR^\ast_Q \tau (v) 
& = & \oR^\ast_Q \operatorname{quot} (v) =
 \operatorname{quot} \oR^\ast (v) =
 \tau \circ \operatorname{proj} \circ \oR^\ast (v) \\
& = & \tau \circ R^\ast_\perp \circ \operatorname{proj} (v)
 = \tau R^\ast_\perp (v).
\end{eqnarray*}
\end{proof}
Using (\ref{equ.rbarisrkplusrperp}) and Lemma \ref{lem.taurperprqcomm},
the diagram
\[ \xymatrix{
H^q (E) \ar@{=}[r] \ar[d]_{\oR^\ast} & K^\perp \oplus K
 \ar[r]^{\tau \oplus \operatorname{id}_K}_\cong 
 \ar[d]_{R^\ast_\perp \oplus R^\ast_K} &
\frac{H^q (E)}{K} \oplus K \ar[d]^{R^\ast_Q \oplus R^\ast_K} \\
H^q (E) \ar@{=}[r] & K^\perp \oplus K
 \ar[r]^{\tau \oplus \operatorname{id}_K}_\cong 
 & \frac{H^q (E)}{K} \oplus K
} \]
commutes. We have constructed an explicit isomorphism
\[ \beta = \tau \oplus \operatorname{id}_K: H^q (E) 
\stackrel{\cong}{\longrightarrow} (H^q (E)/K) \oplus K \] such that
\[ \xymatrix{
H^q (E) \ar@{=}[r]^<<<<{\sim}_<<<<{\beta} \ar[d]_{\oR^\ast} & (H^q (E)/K)\oplus K
  \ar[d]^{R^\ast_Q \oplus R^\ast_K} \\
H^q (E) \ar@{=}[r]^<<<<{\sim}_<<<<{\beta} & (H^q (E)/K) \oplus K
} \]
commutes. Therefore, the action of $\oR^\ast$ on $H^q (E)$ is completely
determined by $R^\ast_Q$ and $R^\ast_K$. \\

The Leray-Serre spectral sequence $\{ (E^{\bullet, \bullet}_r, d_r) \}$ of a 
fibration $f: X\to B$ with fiber $F$ has the following properties. With
$B_p$ the $p$-skeleton of the CW-complex $B$,
$X_p = f^{-1} (B_p),$ and
\[ J^{p,q} = \ker (H^{p+q} (X)\longrightarrow H^{p+q} (X_{p-1})), \]
there is an isomorphism
\[ \alpha (f): J^{p,q} / J^{p+1,q-1} \cong E^{p,q}_\infty \]
and an isomorphism
\[ \gamma (f): H^p (B; \der^q (F)) \cong E^{p,q}_2, \]
where $\der^q (F)$ is the local system on $B$ induced by $f$ with 
group $H^q (F)$ over a point in $B$. Suppose
$\overline{f}: \oX \to B$ is another fibration having fiber, say, 
$\oF$, associated spectral sequence $\{ (\oE^{\bullet, \bullet}_r, 
\overline{d}_r) \}$ and filtration $\oJ^{\bullet, \bullet}$. Suppose
furthermore that $\phi: X\to \oX$ is a fiber preserving map. Then
$\phi (X_p)\subset \oX_p$ and on cohomology
$\phi^\ast (\oJ^{p,q})\subset J^{p,q}$ so that $\phi^\ast$ is
filtration preserving and induces a morphism $\phi^\ast:
\oE \to E$ of spectral sequences. A morphism of spectral sequences
induces a morphism $\phi^\ast_\infty:  \oE_\infty \to E_\infty$
such that the diagram with exact rows
\[ \xymatrix{
0 \ar[r] & \oJ^{p+1,q-1} \ar[r] \ar[d]_{\phi^\ast} &
  \oJ^{p,q} \ar[r] \ar[d]_{\phi^\ast} &
 \oE^{p,q}_\infty \ar[r] \ar[d]_{\phi^\ast_\infty} & 0 \\
0 \ar[r] & J^{p+1,q-1} \ar[r]  &
  J^{p,q} \ar[r]  & E^{p,q}_\infty \ar[r] & 0
} \]
commutes, see \cite[page 49]{mccleary}. Thus $\alpha (f)$ is natural
--- the square
\[ \xymatrix{
\oJ^{p,q} / \oJ^{p+1,q-1} \ar[r]^<<<<{\cong}_<<<<{\alpha (\overline{f})}
 \ar[d]_{\phi^\ast} & \oE^{p,q}_\infty \ar[d]^{\phi^\ast_\infty} \\
J^{p,q} / J^{p+1,q-1} \ar[r]^<<<<{\cong}_<<<<{\alpha (f)}
 & E^{p,q}_\infty 
} \]
commutes. The map $\phi$ induces furthermore a homomorphism
$\phi: \der^q (\oF)\to \der^q (F)$ of local systems, which in turn induces
a map $\phi^\ast: H^p (B; \der^q (\oF))\to H^p (B; \der^q (F)).$
By \cite[Chapter XIII, Theorem 4.9 (6)]{gwwhitehead}, $\gamma$ is
natural: The diagram
\[ \xymatrix{
H^p (B; \der^q (\oF)) \ar[r]^<<<<{\cong}_<<<<{\gamma (\overline{f})}
 \ar[d]_{\phi^\ast} & \oE^{p,q}_2 \ar[d]^{\phi^\ast_2} \\
H^p (B; \der^q (F)) \ar[r]^<<<<{\cong}_<<<<{\gamma (f)}
 & E^{p,q}_2
} \]
commutes.

Let us specialize to our bundle $E$ over the circle $B=S^1$ and the
fiber preserving map $\phi = \oR: E\to E.$ We give the circle its minimal
cell structure. Then the filtration $J$ consists of two possibly nontrivial
pieces,
\[ J^{0,q} = H^q (E),~ J^{1,q-1} =K,~ (J^{2,q-2}=0), \]
and we have
\[ E^{0,q}_\infty \cong \frac{J^{0,q}}{J^{1,q-1}} = \frac{H^q (E)}{K},~
E^{1,q-1}_\infty \cong \frac{J^{1,q-1}}{J^{2,q-2}} = K. \]
Since the base is one-dimensional, only the $p=0$ and $p=1$ columns
can contain nonzero entries. Thus the spectral sequence collapses
at $E_2$ (\emph{not} using our main Theorem \ref{thm.main1})
and $E^{p,q}_2 = E^{p,q}_\infty,$ $\phi_2 = \phi_\infty$
(i.e. $\oR_2 = \oR_\infty$). Putting the above isomorphisms
$\gamma, \alpha, \beta$ together, we obtain a commutative diagram
\[ \xymatrix{
H^0 (S^1; \der^q (F))\oplus H^1 (S^1; \der^{q-1}(F))
  \ar[r]^{\oR^\ast \oplus \oR^\ast}
  \ar[d]_{\cong}^{\gamma \oplus \gamma} &
H^0 (S^1; \der^q (F))\oplus H^1 (S^1; \der^{q-1}(F))
  \ar[d]_{\cong}^{\gamma \oplus \gamma} \\
E^{0,q}_2 \oplus E^{1,q-1}_2 
 \ar[r]^{\oR^\ast_2 \oplus \oR^\ast_2}
 \ar@{=}[d] & 
E^{0,q}_2 \oplus E^{1,q-1}_2
 \ar@{=}[d] \\
E^{0,q}_\infty \oplus E^{1,q-1}_\infty 
 \ar[r]^{\oR^\ast_\infty \oplus \oR^\ast_\infty}
  &  E^{0,q}_\infty \oplus E^{1,q-1}_\infty \\
\frac{H^q (E)}{K} \oplus K \ar[u]^{\cong}_{\alpha \oplus \alpha}
  \ar[r]^{R^\ast_Q \oplus R^\ast_K} &
\frac{H^q (E)}{K} \oplus K \ar[u]^{\cong}_{\alpha \oplus \alpha} \\
H^q (E) \ar[u]^{\cong}_{\beta} \ar[r]^{\oR^\ast} &
H^q (E). \ar[u]^{\cong}_{\beta}
} \]
Using the isomorphisms $\kappa^G$, for $G=\pi_1 (S^1)=\intg,$ and
$\theta$, we obtain a commutative diagram
\[ \xymatrix{
H^q (F)^\intg \oplus \frac{\Der_\intg}{\Princ_\intg} H^{q-1} (F)
\ar[r]^{\oR^\ast \oplus \oR^\ast}
 \ar[d]^{\kappa^\intg \oplus \theta}_{\cong} &
H^q (F)^\intg \oplus \frac{\Der_\intg}{\Princ_\intg} H^{q-1} (F)
 \ar[d]^{\kappa^\intg \oplus \theta}_{\cong} \\
H^0 (S^1; \der^q (F))\oplus H^1 (S^1; \der^{q-1}(F))
  \ar[r]^{\oR^\ast \oplus \oR^\ast} &
H^0 (S^1; \der^q (F))\oplus H^1 (S^1; \der^{q-1}(F)),
} \]
which computes the twisted cohomology terms. Appending the above
two diagrams, we have shown:

\begin{prop} \label{prop.hqe}
Let $F$ be a closed, oriented, path-connected, Riemannian manifold
and $R,S$ two commuting, orientation preserving isometries of $F$.
Let $E$ be the mapping torus of $S$ and $\oR: E\to E$ the fiber
preserving isometry induced by $R$, 
$\oR [t,x] = [t, R(x)],$ $t\in \real^1,$ $x\in F$. Then there is an
isomorphism
\[ H^q (E) \cong H^q (F)^\intg \oplus \frac{\Der_\intg}{\Princ_\intg} H^{q-1} (F), \]
which is independent of $R$, identifies $\Der_\intg H^{q-1} (F)/
\Princ_\intg H^{q-1}(F)$ with $\ker (H^q (E)\to H^q (F)),$ and makes the
diagram 
\[ \xymatrix{
H^q (E) \ar@{=}[r]^<<<<{\sim} \ar[d]_{\oR^\ast} &
H^q (F)^\intg \oplus \frac{\Der_\intg}{\Princ_\intg} H^{q-1} (F)
 \ar[d]^{\oR^\ast \oplus \oR^\ast} \\
H^q (E) \ar@{=}[r]^<<<<{\sim} &
H^q (F)^\intg \oplus \frac{\Der_\intg}{\Princ_\intg} H^{q-1} (F)
} \]
commute.
\end{prop}

With this proposition in hand, we are in a position to describe the 
recursive scheme calculating the equivariant cohomology of
isometric $\intg^n$-actions on $F$.
For notational convenience, we are content with describing the 
method for $n=3$ isometries.
It will be apparent to everyone how to proceed if $n$ is larger.
Let $F$ be a closed, oriented, path-connected, Riemannian manifold.
Let $R_1, R_2, R_3$ be commuting orientation preserving isometries of $F$, 
determining
a $\intg^3$-action on $F$. We wish to calculate the equivariant
cohomology $H^\bullet_{\intg^3} (F)$. Let $B_1$ be the Borel
construction of the $\intg$-action on $F$ by powers of $R_1$, that is,
$B_1$ is the mapping torus $B_1 = \real^1 \times_\intg F$.
The closed, oriented manifold $B_1$ is a fiber bundle over the circle
with fiber $F$. The isometry $R_2$ defines a fiber preserving isometry
$\oR_2: B_1 \to B_1$ by $\oR_2 [t,x] = [t, R_2 (x)].$
Since the monodromy of the mapping torus $B_1$ is $R_1$, the
monodromy of the induced local system $\der^q (F)$ is $R^\ast_1$.
Thus, assuming that the action of $R^\ast_1$ on $H^\bullet (F)$,
defining a $\intg$-action on $H^\bullet (F)$, as well as the action
of $R^\ast_2$ on $H^\bullet (F)$, are
known, the automorphism $\oR^\ast_2$ of $H^q (B_1)$ can be
computed, using Proposition \ref{prop.hqe}, via the diagram
\[ \xymatrix{
H^q (B_1) \ar@{=}[r]^<<<<{\sim} \ar[d]_{\oR^\ast_2} &
H^q (F)^\intg \oplus \frac{\Der_\intg}{\Princ_\intg} H^{q-1} (F)
 \ar[d]^{\oR^\ast_2 \oplus \oR^\ast_2} \\
H^q (B_1) \ar@{=}[r]^<<<<{\sim} &
H^q (F)^\intg \oplus \frac{\Der_\intg}{\Princ_\intg} H^{q-1} (F).
} \]
By the same method, the action of the automorphism induced
on $H^\bullet (B_1)$ by $[t,x] \mapsto [t, R_3 (x)]$ can be
determined.

Let $B_2$ be the Borel
construction of the $\intg$-action on $B_1$ by powers of $\oR_2$, that is,
$B_2$ is the mapping torus $B_2 = \real^1 \times_\intg B_1$.
The closed, oriented manifold $B_2$ is a fiber bundle over the circle
with fiber $B_1$. The isometry $R_3$ defines a fiber preserving isometry
$\oR_3: B_2 \to B_2$ by $\oR_3 [t_1, [t_2 ,x]] = [t_1, [t_2, R_3 (x)]].$
This is well-defined, since $R_3$ commutes with $R_1$ and $R_2$:
For integers $m$ and $n$, we have
\begin{eqnarray*}
\oR_3 [t_1 +n, \oR^n_2 [t_2 +m, R^m_1 (x)]]
& = & \oR_3 [t_1 +n, [t_2 +m, R^n_2 R^m_1 (x)]] \\
& = & [t_1 +n, [t_2 +m, R_3 R^n_2 R^m_1 (x)]] \\
& = & [t_1 +n, [t_2 +m, R^n_2 R^m_1 R_3 (x)]] \\
& = & [t_1 +n, \oR^n_2 [t_2 +m, R^m_1 R_3 (x)]] \\
& = & [t_1, [t_2, R_3 (x)]] \\
& = & \oR_3 [t_1, [t_2, x]].
\end{eqnarray*}
Using Proposition \ref{prop.hqe}, we compute
$\oR^\ast_3: H^\bullet (B_2)\to H^\bullet (B_2)$ by
\[ \xymatrix{
H^q (B_2) \ar@{=}[r]^<<<<{\sim} \ar[d]_{\oR^\ast_3} &
H^q (B_1)^\intg \oplus \frac{\Der_\intg}{\Princ_\intg} H^{q-1} (B_1)
 \ar[d]^{\oR^\ast_3 \oplus \oR^\ast_3} \\
H^q (B_2) \ar@{=}[r]^<<<<{\sim} &
H^q (B_1)^\intg \oplus \frac{\Der_\intg}{\Princ_\intg} H^{q-1} (B_1).
} \]
Here, the $\intg$-action on $H^\bullet (B_1)$ is given by powers of
$\oR^\ast_2: H^\bullet (B_1) \to H^\bullet (B_1),$ which we computed
in the previous step. Since the action of the automorphism induced
on $H^\bullet (B_1)$ by $[t,x] \mapsto [t, R_3 (x)]$ is also known from
the previous step, the right hand side of the diagram can indeed be computed.

Finally, let $B_3$ be the Borel
construction of the $\intg$-action on $B_2$ by powers of $\oR_3$, that is,
$B_3$ is the mapping torus $B_3 = \real^1 \times_\intg B_2$.

\begin{lemma} \label{lem.identifyborelconstr}
The manifold $B_3$ is diffeomorphic to the Borel construction
$F_{\intg^3} = E\intg^3 \times_{\intg^3} F = \real^3 \times_{\intg^3} F$
of the $\intg^3$-action on $F$.
\end{lemma}
\begin{proof}
A smooth map $\phi: B_3 \to F_{\intg^3}$ is given by
\[ \phi [t_1, [t_2, [t_3, x]]] = [(t_1, t_2, t_3), x]. \]
This is well-defined because
\begin{eqnarray*}
\phi [t_1 + n, \oR^n_3 [t_2 + m, \oR^m_2 [t_3 +p, R^p_1 (x)]]]
& = & \phi  [t_1 + n, \oR^n_3 [t_2 + m, [t_3 +p, R^m_2 R^p_1 (x)]]] \\
& = & \phi  [t_1 + n, [t_2 + m, [t_3 +p, R^n_3 R^m_2 R^p_1 (x)]]] \\
& = & [(t_1 +n, t_2 +m, t_3 +p), R^n_3 R^m_2 R^p_1 (x)] \\
& = & [(t_1, t_2, t_3), x] \\
& = & \phi [t_1, [t_2, [t_3, x]]].
\end{eqnarray*}
A smooth map $\psi: F_{\intg^3} \to B_3$ is given by
\[ \psi [(t_1, t_2, t_3), x] = [t_1, [t_2, [t_3, x]]], \]
which is also well-defined as
\begin{eqnarray*}
\psi [(t_1 +n, t_2 +m, t_3 +p), R^n_3 R^m_2 R^p_1 (x)]
& = & [t_1 + n, [t_2 + m, [t_3 +p, R^n_3 R^m_2 R^p_1 (x)]]] \\
& = & [t_1 + n, \oR^n_3 [t_2 + m, \oR^m_2 [t_3 +p, R^p_1 (x)]]] \\
& = & [t_1, [t_2, [t_3, x]]] \\
& = & \psi [(t_1, t_2, t_3), x]. 
\end{eqnarray*}
The maps $\phi$ and $\psi$ are inverses of each other.
\end{proof}
Then the equivariant cohomology of $F$ is given by
\[ H^q_{\intg^3} (F) = H^q (F_{\intg^3}) \cong H^q (B_3) \cong
  H^q (B_2)^\intg \oplus \frac{\Der_\intg}{\Princ_\intg} H^{q-1} (B_2), \]
where the $\intg$-action on $H^\bullet (B_2)$ is given by powers
of $\oR^\ast_3$, computed in the previous step.

\providecommand{\bysame}{\leavevmode\hbox to3em{\hrulefill}\thinspace}
\providecommand{\MR}{\relax\ifhmode\unskip\space\fi MR }
\providecommand{\MRhref}[2]{%
  \href{http://www.ams.org/mathscinet-getitem?mr=#1}{#2}
}
\providecommand{\href}[2]{#2}


\begin{thebibliography}{LRV03}

\bibitem[AM94]{ademmilgram}
A.~Adem and J.~Milgram, \emph{Cohomology of finite groups}, Grundlehren der
  math. {W}issenschaften, vol. 309, Springer Verlag Berlin Heidelberg, 1994.

\bibitem[Ban10]{banagl-intersectionspaces}
M.~Banagl, \emph{Intersection spaces, spatial homology truncation, and string
  theory}, Lecture Notes in Math., vol. 1997, Springer Verlag Berlin
  Heidelberg, 2010.

\bibitem[Ban11]{banagl-derhamintspace}
\bysame, \emph{Foliated stratified spaces and a de {R}ham complex describing
  intersection space cohomology}, submitted preprint, arxiv:1102.4781, 2011.

\bibitem[Bre97]{bredon}
G.~E. Bredon, \emph{Sheaf theory}, second ed., Grad. Texts in Math., no. 170,
  Springer Verlag, 1997.

\bibitem[BT82]{botttu}
R.~Bott and L.~W. Tu, \emph{Differential forms in algebraic topology}, Graduate
  Texts in Mathematics, no.~82, Springer Verlag, 1982.

\bibitem[Dai91]{dainonmult}
X.~Dai, \emph{Adiabatic limits, nonmultiplicativity of signature, and {L}eray
  spectral sequence}, J. Amer. Math. Soc. \textbf{4} (1991), no.~2, 265 -- 321.

\bibitem[DH89]{davishausmann}
M.~W. Davis and J.-C. Hausmann, \emph{Aspherical manifolds without smooth or
  {PL} structure}, Algebraic Topology, Lecture Notes in Math., vol. 1370,
  Springer Verlag, Berlin-Heidelberg, 1989, pp.~135 -- 142.

\bibitem[Hat04]{hatcherss}
A.~Hatcher, \emph{Spectral sequences in algebraic topology}, Preprint, 2004.

\bibitem[HH84]{hoffmanhomer}
M.~Hoffman and W.~Homer, \emph{On cohomology automorphisms of complex flag
  manifolds}, Proc. Amer. Math. Soc. \textbf{91} (1984), no.~4, 643 -- 648.

\bibitem[LRV03]{lueckreichvarisco}
W.~L{\"u}ck, H.~Reich, and M.~Varisco, \emph{Commuting homotopy limits and
  smash products}, K-Theory \textbf{30} (2003), 137 -- 165.

\bibitem[L{\"u}c07]{lueckratcompkdiscgrps}
W.~L{\"u}ck, \emph{Rational computations of the topological {K}-theory of
  classifying spaces of discrete groups}, J. Reine Angew. Math. \textbf{611}
  (2007), 163 -- 187.

\bibitem[McC85]{mccleary}
J.~McCleary, \emph{User's guide to spectral sequences}, Mathematics Lecture
  Series, no.~12, Publish or Perish, 1985.

\bibitem[Mil58]{milnorconncurvzero}
J.~Milnor, \emph{On the existence of a connection with curvature zero},
  Comment. Math. Helv. \textbf{32} (1958), 215 -- 223.

\bibitem[MS74]{milnorstasheff}
J.~W. Milnor and J.~D. Stasheff, \emph{Characteristic classes}, Ann. of Math.
  Studies, vol.~76, Princeton Univ. Press, Princeton, 1974.

\bibitem[M{\"u}l11]{muellerjoern}
J.~M{\"u}ller, \emph{A {H}odge-type theorem for manifolds with fibered cusp
  metrics}, preprint, 2011.

\bibitem[Pal61]{palaisexistslices}
R.~S. Palais, \emph{On the existence of slices for actions of non-compact {L}ie
  groups}, Ann. of Math. \textbf{73} (1961), no.~2, 295 -- 323.

\bibitem[Smi77]{smillie}
J.~Smillie, \emph{Flat manifolds with non-zero {E}uler characteristics},
  Comment. Math. Helv. \textbf{52} (1977), 453 -- 455.

\bibitem[Whi78]{gwwhitehead}
G.~W. Whitehead, \emph{Elements of homotopy theory}, Graduate Texts in Math.,
  no.~61, Springer Verlag, 1978.

\end{thebibliography}
\end{document}